\documentclass[12pt,letterpaper]{amsart}

\usepackage[makeindex]{imakeidx}
\usepackage{amsmath,amsthm, amsbsy,amsfonts,amssymb, txfonts,tabularx,booktabs, wasysym}
\usepackage{euscript} %added by Tathagata

\usepackage[normalem]{ulem}
\usepackage{stmaryrd,mathrsfs,graphicx, supertabular, amscd, tikz-cd, tikz}

\usetikzlibrary{graphs}
\usetikzlibrary{graphs.standard}
\usepackage{stackrel}
\usetikzlibrary{matrix,arrows,decorations.pathmorphing}

\usepackage[active]{srcltx}
\allowdisplaybreaks
\numberwithin{equation}{section}

\usepackage[
	hypertexnames=false,
	hyperindex,
	pagebackref,
	pdftex,
	breaklinks=true,
	bookmarks=false,
	colorlinks,
	linkcolor=blue,
	citecolor=blue,
	urlcolor=blue,
]{hyperref}
\usepackage{hyperref}

\def\Art{\mathop{\rm Art}}
\def\Ps{\EuScript{P}}
\def\semidirect{\rtimes}

\def\C{\mathbb{C}}

\def\Cs{\EuScript{C}}
\def\Ds{\EuScript{D}}
\def\Es{\EuScript{E}}
\def\Is{\EuScript{I}}
\def\Ls{\EuScript{L}}
\def\Rs{\EuScript{R}}
\def\M{\EuScript{M}}
\def\N{\EuScript{N}}
\def\Ts{\EuScript{T}}
\def\Xs{\EuScript{X}}

\def\Ss{\EuScript{S}}
\def\iso{\cong}

\def\Z{\mathbb{Z}}

\def\Aut{\mathop{\rm Aut}}

\def\AA{{\mathbb A}}
\def\BB{{\mathbb B}}

\def\CC{{\mathbb C}}

\def\FF{{\mathbb F}} 
\def\GG{{\mathbb G}}

\def\PP{{\mathbb P}}

\def\QQ{{\mathbb Q}} 
\def\RR{{\mathbb R}}

\def\ZZ{{\mathbb Z}}

\def\ssm{\smallsetminus}

\def\cub{\mathrm{cub}}

\def\G{\Gamma}
\def\g{\gamma}

\def\orb{{\rm orb}}

\def\red{{\rm red}} 
\def\reg{{\rm reg}}

\def\int{{\rm int}}

\def\st{{\rm st}}
\def\sst{{\rm sst}}
\def\sing{{\rm sing}}

\def\bs{\backslash}
\newcommand{\eps}{\varepsilon}

\newcommand{\p}{\partial}

\def\Ecal{{\mathcal D}}
\def\Ecal{{\mathcal E}}

\def\Lcal{{\mathcal L}}
\def\Mcal{{\mathcal M}}

\def\Ocal{{\mathcal O}}
\def\Pcal{{\mathcal P}}

\def\Vcal{{\mathcal V}}

\def\Escr{{\mathscr E}}

\def\Rscr{{\mathscr R}}

\def\la{\langle}
\def\ra{\rangle}

\def\Sfrak{\mathfrak{S}}

\newcommand\ar{\operatorname{Art}}

\newcommand\aut{\operatorname{Aut}}

\newcommand\bl{\operatorname{Bl}}

\newcommand\pic{\operatorname{Pic}}

\newcommand\sym{\operatorname{Sym}}
\newcommand\GL{\operatorname{GL}}

\newcommand\Ind{\operatorname{Ind}}
\newcommand\Pop{\operatorname{P}}
\newcommand\PGL{\operatorname{PGL}}

\newcommand\SL{\operatorname{SL}}

\newcommand\Star{\operatorname{Star}}

\newcommand\Tcal{\operatorname{{\mathcal T}}}

\newcommand\Orth{\operatorname{O}}

\newcommand\PSL{\operatorname{PSL}}
\newcommand\U{\operatorname{U}}

\newcommand\PU{\operatorname{PU}}
\newcommand\Sp{\operatorname{Sp}}

\newcommand{\abs}[1]{\lvert#1\rvert}

\newtheorem{theorem}{Theorem}[section]
\newtheorem{lemma}[theorem]{Lemma}
\newtheorem{proposition}[theorem]{Proposition}
\newtheorem{corollary}[theorem]{Corollary}

\theoremstyle{definition} 
\newtheorem{definition}[theorem]{Definition}

\theoremstyle{remark}

\newtheorem{remark}[theorem]{Remark}

\newtheorem{question}[theorem]{Question}

\newtheorem{notation}[theorem]{Notation}

\title[Cubic surfaces moduli space and its fundamental group]{The moduli space of cubic surfaces and its orbifold fundamental group}
\author{Daniel Allcock}
\address{Department of Mathematics\\University of Texas at Austin\\ Austin, TX 78712}
\author{Tathagata Basak}
\address{Department of Mathematics\\Iowa State University \\ Ames, IA 50011}
\author{Eduard Looijenga}
\address{Mathematisch Instituut, Universeit Utrecht, Nederland}
\thanks{TB was partially supported by Simons Foundation Collaboration Grant 637005.}
\begin{document}
\begin{abstract} 
We find  a new presentation of the orbifold fundamental group of the moduli space of cubic surfaces, based on the Artin group of the Petersen graph.
    This group has
     two well-known quotients, which appear in the simplest possible way:  
    killing the squares resp.\ sixth powers of our generators yields
     the Coxeter group  $W(E_6)$ resp.\ a certain arithmetic  group in $\PU(1,4)$.
    Along the way we present a certain central extension by~$\Z$, of the genus zero mapping class group
    with any number of punctures.  The $6$-puncture case is essential
    in this paper, and the general case may have independent interest.
\end{abstract}
\maketitle
\begin{center}
\emph{Dedicated to the memory of Domingo Toledo (1945-2025).}   
\end{center}
\section{Introduction}

Smooth cubic surfaces  are parametrized by  a $4$-dimensional variety $\M_\cub$. This variety is an orbifold (a Deligne-Mumford stack), and our main goal is to give a pretty
and very symmetric finite presentation of its orbifold fundamental group.  
%xx it was shown by Allcock \cite[Cor.\ 3.4]{A00} to be aspherical 
Symmetric presentations must exist, by choosing a basepoint representing
a cubic surface with many symmetries.
We use the Clebsch cubic surface $X_C$, 
defined in~$P^4$ by
$$
X_0+\cdots +X_4=X_0^3+\cdots+X_4^3=0.
$$
Its  automorphism group  $\Sfrak_5$ is a subgroup 
of $\pi_1^\orb(\M_\cub,[X_C])$, so there must be a presentation manifesting this 
$\Sfrak_5$ symmetry.
We will describe 
$\pi_1^\orb(\M_\cub,\discretionary{}{}{}[X_C])$ as a quotient of the Artin group of the Petersen graph~$\Ps$ (see figure \ref{fig-Petersen}):
\begin{figure}
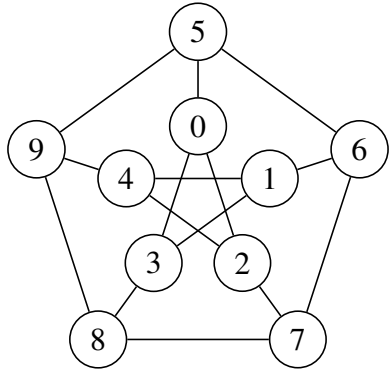

\tikzgraphsset{edges={draw,semithick}, nodes={circle,draw,semithick}}
 \tikz \graph[math nodes, clockwise]
    { subgraph I_n [V={0,1,2,3,4}] --
      subgraph C_n [V={5,6,7,8,9},radius=1.25cm];
      {[cycle] 0,2,4,1,3} };
      \caption{The Petersen graph.}
      \label{fig-Petersen}
\end{figure}
The Artin relations, and the additional relations needed to obtain
$\pi_1^\orb(\M_\cub)$, are invariant 
under the automorphism group $\Sfrak_5$ of~$\Ps$.

The Artin group of a graph has one generator~$\sigma_\nu$ for each vertex~$\nu$, with two
generators $\sigma_v, \sigma_{v'}$ braiding 
($\sigma_v\sigma_{v'}\sigma_v=\sigma_{v'}\sigma_v\sigma_{v'}$) or commuting
according to whether or not $v$ and $v'$ are joined.
Symmetries of the graph obviously act on the Artin group.
So we may form the semidirect product $\Art(\Ps)\rtimes \Aut(\Ps)$. 
On this we impose an $\Aut(\Ps)$-orbit of relations, or equivalently any one of
them.  
Choose a vertex $v$ of $\Ps$; the vertices not connected to it form a hexagon.
Label its vertices $v_{i\in\ZZ/6}$ cyclically in either direction.
For example,  one could take $v=0$ in the figure above,  
and then $(v_0,\ldots, v_5)=(1,4,9,8,7,6)$.  One can check that 
there is a unique $\rho\in \aut(\Ps)$ which fixes $v$ and 
sends each $v_i$ to $v_{i+1}$.  The relation associated to  $\nu,\nu_0,\dots,\nu_5$ is
\begin{equation}
    \label{EqRelation}
    \rho\sigma_{v}\equiv\sigma_{v_0}\sigma_{v_1}\sigma_{v_2}\sigma_{v_3}\sigma_{v_4}.
\end{equation}
(Although $\sigma_{\nu_5}$ does not appear explicitly, it appears in
the definition of~$\rho$.)

\begin{theorem}[Main theorem]
    Let $G(\Ps)$ be
    the quotient of $\Art(\Ps)\semidirect\Aut(\Ps)$ by any one relation
    \eqref{EqRelation}, or equivalently by all of them.
    Then 
    $$G(\Ps)\iso \pi_1^\orb(\M_\cub).$$  Furthermore, $\Art(\Ps)$
    surjects to~$G(\Ps)$, hence also to
    $\pi_1^\orb(\M_\cub)$.
\end{theorem}

In fact theorem~\ref{thm:main} exhibits a specific isomorphism
$G(\Ps)\to\pi_1^\orb(\M_\cub)$. We will see that there is precisely one conjugacy class of subgroups of $G(\Ps)$ isomorphic to $\Sfrak_5$ (see Remark \ref{rem:unique-S5}). This suggests that our presentation of this orbifold fundamental group is a very natural one.
The $\sigma_\nu$ are obviously conjugate in~$\Art(\Ps)\semidirect\Aut(\Ps)$,
so they distinguish a conjugacy class in~$\pi_1^\orb(\M_\cub)$.
Our proof shows that this  class is represented  by simple loops around the discriminant which are often referred to as \emph{meridians)}. 

We elaborate, because the discriminant plays a central role in the paper.
We recall that the cubic surfaces that are stable in the sense of geometric invariant theory are precisely  those that have  nodal singularities (=ordinary double points) at worst. Their moduli space $\M_\cub^\st$ is also an orbifold, 
and contains $\M_\cub$ as an open dense subset.
The discriminant $\Ds_\cub$ means the
complement of~$\M_\cub$, and is an irreducible normal crossing divisor in the orbifold sense.
By this irreducibility, the simple loops  around 
$\Ds_\cub$ form a single conjugacy class, which by construction contains our generators.

The normal crossing property implies that if we impose 
on $\pi_1^\orb(\M_\cub)$
a further relation by stipulating that for a given positive integer $\ell$,  the $\ell$th power of the image of $\sigma_v$ is trivial, then
there exists a ramified covering in the analytic orbifold setting:
\[
(\M^\st_\cub[\ell], \Ds_\cub[\ell])\to (\M^\st_\cub, \Ds_\cub)
\]
whose covering group is the quotient $G(\Ps)_\ell$ of $G(\Ps)$ defined by this relation.  
According to Proposition \ref{prop:squares of generators}, 
$\M^\st_\cub[2]$ is the moduli space of marked stable cubic surfaces, $\Ds_\cub[2]$ its discriminant and $G(\Ps)_2$ a Coxeter group of type~$E_6$.  
Proposition \ref{prop:cubes of generators} provides a similar interpretation for  $\ell=6$ by  linking  this with the ball quotient description of $\M^\st_\cub$ 
found by Allcock-Carlson-Toledo \cite{act:cubic_surfaces}:  
$\M^\st_\cub[6]$ is a complex ball on which 
$G(\Ps)_6$ acts as an arithmetic group. 
Our  generators acts on that ball as order~$6$ complex reflections, and 
$\Ds_\cub[6]$ is the union of their mirrors (fixed point sets).  
In Proposition~\ref{prop:G3} we identify the (finite) deck group of
$\M^\st_\cub[3]$.
This suggests:

\begin{question}
    Do the spaces $\M_\cub^\st[\ell=3,4,5]$ have
    special interpretations?
\end{question}

The idea that
$\pi_1^\orb(\M_\cub)$ is a quotient of $\Art(\Ps)$ came from
the ball quotient model of $\M_\cub^\st$.  Namely, $X_C$
corresponds to a point of the ball with stabilizer~$\Sfrak_5$, 
and the Sylvester polytope (see below) manifests as a real hyperbolic polytope
centered there.  The braid and commutation relations of~$\Art(\Ps)$ hold for
geometric reasons.
However, we establish Theorem~\ref{thm:main}
without reference to the ball model and its proof draws rather on an argument of Lefschetz type.

Our special interest in $\pi_1^\orb(\M_\cub)$ as a quotient
of~$\Art(\Ps)$ derives from its role as  an analogue of Allcock's ``monstrous proposal''
\cite{allcock:proposal}.  There, a larger ball replaces the $4$-ball, and a larger Artin group replaces $\Art(\Ps)$.  The Artin generators satisfy a relation  very much like \eqref{EqRelation} \cite{AB25}. The
monstrous proposal asserts that the analogue of
$G(\Ps)_2$ is $(M\times M)\semidirect \Z/2$, where $M$ is the
monster simple group. This proposal was also the motivation of \cite{loo23}, which among other things provided substantial evidence for  this proposal. The Lefschetz argument alluded to above was suggested by this paper.

Another presentation of the orbifold fundamental group had been given a while ago by one of us (E.L., \cite{looijenga:affartin}), based on a different construction. This  led to a different presentation, but has nevertheless  of the same flavor,  with the affine Coxeter diagram of type $\hat E_6$ taking the role of the Petersen graph. We have not formally derived one from the other, although this seems doable. After all, $\Ps$ contains  the $\hat E_6$ Coxeter diagram  as a full subgraph.

\medskip
A notable  ingredient  of our proof is a description of the 
discriminant $\Ds_\cub[2]$ of marked cubic surfaces, 
entirely in terms of a finite Coxeter group of type $E_6$ without 
any reference to cubic surfaces.  In fact,  we do better than that, for we also make such a construction  for the preimage of $\Ds_\cub$ (which we denote by $\N$) under  the real oriented blowup $\M^+_\cub\to \M^\st_\cub$ of $\Ds_\cub$ in $\M^\st_\cub$. The space $\M^\st_\cub$ is an orbifold version of a manifold with corners and so $\N$ has (at the very least) the structure of
a stratified PL-manifold.  It is not a surprise that the  inclusion  $\M_\cub\subset \M^+_\cub$ is a homotopy equivalence in the orbifold setting and we establish early on that the inclusion $\N=\p \M^+_\cub \subset \M^+_\cub$ induces isomorphisms on orbifold fundamental groups. This reduces our task to finding  a presentation of the orbifold fundamental group of the more concrete $\N$. 

Since  $\N$  arises as the boundary of a  manifold with corners, it is naturally stratified. It has only one open stratum $\mathring \N$ and  this stratum  comes as a  circle bundle over an orbifold familiar to algebraic geometers, namely 
$\Sfrak_6\bs \M_{0,6}$: the space of  $\PGL_2(\CC)$-orbits of $6$-element subsets of $\PP^1$.
Hence the orbifold fundamental group of $\mathring \N$ is a central extension by $\ZZ$ of the mapping class group of a pair $(\PP^1,A)$ with $A\subset \PP^1$ a $6$-element subset a  (this is almost a spherical braid group on six strands). The  orbifold fundamental group of $\N$ is a quotient of this central extension, due to identifications on the boundary.

This approach, concrete as  it may be, puts the base point in $\mathring \N$. This is quite different from 
$[X_C]$ and we thus seem to lose the desired $\aut(\Ps)\cong \Sfrak_5$-symmetry. Let us briefly explain here how 
we restore  this (where we change the details a bit for the sake of brevity).  

This involves a $4$-dimensional polytope $\Pi$, which we call for good reasons,  the \emph{Sylvester polytope}. 
Its  barycenter  is the  unique fixed point of  $\aut(\Pi)$. It  is also a manifold with corners except at a single  $\aut(\Pi)$-orbit of vertices,  denoted  $\Pi_\infty$.  Furthermore, it comes with an identification of $\aut(\Pi)$ with $\aut(\Ps)$.  
We then define a proper  embedding 
$\aut(\Ps)\bs \mathring \Pi$ in $\M_\cub$  with the barycenter mapping  to $[X_C]$. This map extends continuously  to 
\[
\aut(\Ps)\bs (\Pi\ssm \Pi_\infty)\to \M^+_\cub
\]
with $\aut(\Ps)\bs (\p\Pi\ssm \Pi_\infty)$ mapping to $\p\M^+_\cub=\N$. The space $\p\Pi\ssm \Pi_\infty$ is simply connected (it is a $3$-sphere minus a finite set)
% could also fix by replacing 3 by 2:
%(it has the homotopy type of a wedge of $3$-spheres)
and comes with a $\aut(\Ps)$-action. We then let   the orbifold map $\aut(\Ps)\bs \p\Pi\ssm \Pi_\infty\to \N$ take the role of base point and this brings in the desired $\aut(\Ps)$-symmetry.

\smallskip

A glance at the table of contents should make it clear how  this  paper is organized. We therefore only mention here that the Appendix stands on its own: it is independent of the rest of the paper and might be of interest to another  audience as well, as it is about the spherical  braid groups and related matters. We recover Birman's presentation in a way that suits our purpose best (it has more  of an algebro-geometric flavor and is also brief).  We introduce an infinite cyclic central extension of such groups (geometrically defined by an orbifold line bundle) and give  a presentation for them. Only the special case of $6$ strands  matters (but it also crucial) for this paper.

This paper contains no AI-generated text or mathematics.
\tableofcontents

\section{The  stratification of the moduli space of cubic surfaces}

\subsection{Marked cubic surfaces}
 If $X$ is a smooth cubic surface, then $H_2(X)$  is a unimodular odd  lattice  of signature $(1,6)$ and its hyperplane class $\kappa_X\in H_2(X)$ (which is also the 
 anticanonical
 class) satisfies $\kappa_X\cdot\kappa_X=3$.
This prompts us to fix  a unimodular odd  lattice $L$ of signature $(1,6)$ and a vector $\kappa\in L$ with $\kappa\cdot\kappa=3$. It is known that the pair  $(L, \kappa)$ is unique up to isomorphism. Indeed, there exists an orthogonal  basis $\ell, e_1, \dots, e_6$ of $L$ with  $\ell\cdot \ell =1$ and $e_i\cdot e_i=-1$  and 
we can then take $\kappa=3\ell-e_1-\cdots  -e_6$.  So this pair is isomorphic to  $(H_2(X), \kappa_X)$. The choice of such an isomorphism is called  an \emph{$L$-marking} of $X$.

We review of a few properties of the pair $(L, \kappa)$ (see for example \cite{sss} and \cite{dolgachev:classical}). 
For each $\alpha\in L$ perpendicular to $\kappa$ with $\alpha\cdot \alpha=-2$, the orthogonal reflection in $\alpha$ is given by $s_\alpha: x\in L\mapsto x+ (\alpha\cdot x)\alpha$ and hence takes $L$ to itself. Therefore, such $\alpha$ make up  a root system $\Rscr$ whose Weyl group $W$ is the subgroup of
$\Orth(L)$ generated by the reflections $s_\alpha$. It is of type $E_6$: in terms of the $\ell, e_1, \dots, e_6$, 
a root basis for $\Rscr$ is $(\ell-e_1-e_2-e_3, e_1-e_2,\dots , e_5-e_6)$. The positive roots for this basis are $\{e_i-e_j\}_{1\le i<j\le 6}$ ($15$ of those), $\{\ell-e_i-e_j-e_k\}_{1\le i<j<k\le 6}$ ($20$) and
$2\ell-(e_1+\cdots e_6)$.

We write $\Escr\subset L$ for the set $e\in L$  satisfying both $e\cdot e=-1$ and $e\cdot \kappa=1$ (the ``exceptional classes''). 
In terms of the above basis,  this set consists of the $e_i$ (6), $\kappa+e_i-\ell$ (6) and $\ell-e_i-e_j$ with $i\not=j$ (15). An $L$-marking of a cubic surface $X$ takes $e\in L$ to the class of a line if and only if $e\in \Escr$.  The lines  are also exceptional curves of the first kind. 

In particular, the images of $e_1, \cdots ,e_6$ are uniquely represented by six pairwise disjoint lines on $X$. Furthermore,
when contracted they produce a copy  of a projective plane, thus identifying  $X$ with the blowup of that projective  plane in 6 points in general position (no three on a line, not all on a conic) and $\ell$  with the class of a line in this plane.
%the pull-back of the line class of this plane. 

Conversely,  the blowup of six numbered points in $\PP^2$ in general position produces a marked cubic surface: the  linear system of cubic curves passing though these 6 points is of dimension $3$ and realizes the blowup as a cubic surface. The construction produces the marking.

The   Weyl group $W$ is  also the stabilizer of $\kappa$ in the orthogonal group of $L$ and so the $L$-markings of a smooth cubic surface are simply transitively permuted by $W$. 
Both $\Rscr$ and $\Escr$ are  $W$-orbits.  The $W$-stabilizer of an element of $\Rscr$ resp.\ $\Escr$ is  a reflection subgroup of type $A_5$ resp.\ $D_5$.

\subsection{Marked stable cubic surfaces}\label{subsect:marked stable} 
Recall that a  cubic surface in $\PP^3$ is stable (relative to the action of $\SL_4(\CC)$) if and only if each of its singular points is a node  (an $A_1$-singularity). The maximum number of such points  is four and when this happens, the surface is isomorphic  with  the Cayley nodal surface (defined by $z_0z_1z_2+z_1z_2z_3+z_2z_3z_0 +z_3z_0z_1$).
Let $\M^\st_\cub$  stand for the  moduli space of stable cubic surfaces given up to a projective equivalence,  with $\M_\cub\subset \M^\st_\cub$ parametrizing the smooth ones. The variety  $\M^\st_\cub$  is defined over $\RR$ (even over $\QQ$, but this will not matter in what follows) with  $\M^\st_\cub(\RR)$ parametrizing the stable cubic surfaces that can be defined over $\RR$. It also has the structure of an $4$-dimensional orbifold: it has only quotient singularities.  The \emph{discriminant}  in $\M^\st_\cub$ is by definition  $\Ds_\cub:= \M^\st_\cub\ssm \M_\cub$ and is a hypersurface of $\M_\cub$ defined over $\RR$.

Let $\M_\cub(L)$ stand for the  moduli space of  $L$-marked smooth cubic surfaces given up to projective equivalence (in an evident sense). Concretely,  we can identify $\M_\cub(L)$ with the set of pairs $(p_5,p_6)\in \PP^2\times\PP^2$ which together  with the  standard projective basis elements $p_1=[1:0:0]$, $p_2=[0:1:0]$, $p_3=[0:0:1]$, $p_4=[1:1:1]$ form a $6$-tuple in general position. This is a Zariski open subset of $\PP^2\times\PP^2$. The variety $\M_\cub(L)$ and this identification are defined over $\RR$. Let
\[
\M_\cub(L)\subset \M^\st_\cub(L)
\]
be  the normalization of $\M^\st_\cub$ in $\M_\cub(L)$. This normalization is also defined over $\RR$.
A point of $\M^\st_\cub(L)$ has still a geometric  interpretation: it is representable by a stable cubic surface $X$ plus a surjection $L\twoheadrightarrow  H_2(X)$ which factors as 
\begin{equation}\label{eqn:marking}
L\xrightarrow{\cong} H_2(\tilde X)\xrightarrow{\pi_*} H_2(X),
\end{equation}
where $\pi: \tilde X\to X$ is the resolution of singularities obtained by a simple blowup of each singular point.
Since we will encounter this situation later in the paper,  it  is worth recalling here what happens with $\Escr$.
This has been analyzed by Schl\"afli and Cayley in the 19th century and  a modern treatment can be found in  Bruce-Wall \cite{BW79}.  What follows is a brief review of this, adapted to our needs.

If  $X$ has a single node $x$, then this determines  a reflection $s\in W$ which can be understood as  the monodromy  transformation around $X$ as acting on the homology of a nearby smooth cubic surface.
The marking \eqref{eqn:marking} also factors as $L\twoheadrightarrow L_s\cong H_2(X)$, where the first map forms the lattice  of $s$-coinvariants. The kernel of the marking is spanned by a root $\alpha\in\Rscr$, which is then determined up to a sign and we have  $s=s_\alpha$. Such a root is the  called the associated \emph{vanishing cycle}. The factorization \eqref{eqn:marking} is unique  up to precomposition with $s$, but once it is given we  can take $\alpha$ to be the class of the exceptional divisor (a copy of $\PP^1$ with self-intersection $-2$).
\smallskip

The reflection $s$ fixes precisely $15$ elements of $\Escr$ and these define as many distinct lines in $X$, none of which passes through $x$. The $12$ remaining ones, denoted $\Escr_s$,  come in $6$ pairs exchanged by $s$. Each pair determines a line  on $X$ through $x$ and an isomorphism $L\cong H_2(\tilde X)$ as above takes this pair to the classes of the strict and the total transform of  this line (the two are indeed exchanged by $s$).

There are no other lines on $X$ and so the image of $\Escr$ in $L_s$ still labels the lines in $X$ in a bijective manner (\footnote{This is why we prefer to mark the homology rather than the cohomology.}).  
This is geometrically realized as follows.
Denote  projectivized tangent space  of the ambient $\PP^3$ at $x$ by $\PP_x$ (so that is a projective plane). If $x\in X$ is a node, then  the projectivized tangent cone of $X$ at $x$ is a plane conic $C_x\subset\PP_x$  and the six lines on $X$ through $x$ define a six-element subset $D_x$ of $C_x$. We regard this as a positive degree $6$ divisor on $C_x$.

We now turn to the case of an arbitrary  stable cubic surface $X$.  We then  $D_x$ let be the positive divisor on $C_x$ for which a point of $C_x$ has multiplicity $\mu$ if it defines a line on $X$ through $x$ which contains exactly  $\mu$ nodes ($x$ included).

For two distinct nodes $x$ and $x'$ on $X$ the line $\overline{xx'}$  spanned by them  meets $X$ in two distinct points 
with multiplicity $\ge 2$ and hence must be contained  in $X$. It thus  defines a point $p_{x'}\in C_x$. Let $\Pcal$ be the pencil of planes in the ambient $\PP^3$ through $\overline{xx'}$. 
This becomes in the projective plane $\PP_x$  the pencil of lines through $p_{x'}$. Since $p_{x'}\in C_x$, 
this identifies   $\Pcal$ with $C_x$.
Similarly $\Pcal\cong C_{x'}$ and so we obtained  an  isomorphism  $C_x\cong C_{x'}$.

\begin{lemma}\label{lemma:linesonstablecubic}
The divisor $D_x$ has multiplicities $\le 2$ and the isomorphism $C_x\cong C_{x'}$ identifies $D_x$ with $D_{x'}$. In particular, the  isomorphism type of the pair $(C_x, D_x)$ is independent of the node.

If $X$ has been $L$-marked, then the  reflections in $L$ associated with  $x$ and $y$ commute.
\end{lemma} 
\begin{proof}
For  $P\in \Pcal$, its  intersection with  $X$ is a cubic divisor which has $\overline{xx'}$ as a component and hence the residual divisor $Q_P$ is a  quadric.  By Bertini's theorem,  $Q_P$ is  smooth for general $P$.  It must  pass through $x$ and $x'$ (for these points have multiplicity $>1$ on $P\cap X$)  and so  $Q_P\cap \overline{xx'}=\{x,x'\}$.  In particular, there cannot  be  third node on $\overline{xx'}$.

For $P\in\Pcal$ arbitrary,  $Q_P$  still  passes through $x$ and $x'$, but cannot 
be $\ell$ with multiplicity $2$, for then $P$ would meet $X$ with multiplicity $3$ and then $X$ would  be singular along $\ell$. So if $Q_P$  is singular, then it is of the form $\ell_x+\ell_{x'}$ with  $\ell_x$, $\ell_{x'}$ distinct lines on $X$ with   $x\in \ell_x$ and $x'\in \ell_{x'}$. If $x''$ is a node on $\ell_x$ different from $x$, then
the line through $x''$ and $x'$ lies on $X$ and hence must be $\ell_{x'}$. 
This proves that the isomorphism $C_x\cong C_{x'}$ identifies  $D_x$ with $D_{x'}$.

The last clause follows from the fact that the vanishing cycles associated with $x$ and $y$ are disjoint and are therefore perpendicular with respect to the intersection pairing.
\end{proof}

We will later see that $(C_x, D_x)$ is a complete invariant of the pair $(X,x)$. So the automorphism group of $X$ acts transitively on the set of its nodes.

\begin{remark}%\label{rem:}
The number  lines on a stable cubic surface still sums up to  $27$, provided  a line containing one node resp.\ two nodes  is given multiplicity $2$ resp. $4$.  For example, for the Cayley cubic surface,  its four singular points make up a complex tetrahedron whose six edges give as many  lines on it.  The six  lines that normally pass through a  node have now coalesced into three pairs, and hence we are left with $27-6.4=3$ lines lying on the smooth part (they  connect opposite sides of the tetrahedron).    
\end{remark}

It is also known that the discriminant
\[
\Ds_\cub(L) := \M_\cub^\st(L)\ssm \M_\cub(L)
\]  
is a simple normal crossing divisor. 
The  reflections in $W$ bijectively label the  irreducible components  of $\Ds_\cub(L)$ by 
assigning  to a reflection in $W$ the codimension  one part of its  fixed point set in $\Ds_\cub(L)$.
A point  of $\Ds_\cub(L)$ where four irreducible components of $\Ds_\cub(L)$  meet represents the 
Cayley cubic surface endowed with a marking in the above sense.

\subsection{A real oriented blowup} We here leave the algebraic category in the strict sense and enter 
the logarithmic category.  We recall the \emph{real oriented blowup} $M^+\twoheadrightarrow  M$ of a normal crossing divisor $D$ in  
a complex manifold $M$.
If $(U; z_1, \dots, z_n)$ is a holomorphic chart of  $M$ such that $D\cap U$ is given by $z_1z_2\cdots z_k=0$,
then an open subset $U^+$ of $M^+$ is
parametrized by an open subset of $ (\RR_{\ge 0}\times S^1)^k\times\CC^{n-k}$ such that the projection $U^+\twoheadrightarrow U$ is given by
$(r_1, u_1,\dots , r_k, u_k, z_{k+1}, \dots, z_n)\mapsto(r_1u_1, \dots, r_ku_k,  z_{k+1}, \dots, z_n)$. Note that $M^+$ is a manifold with corners, having $M\ssm D$ as its interior.

Over the locus where $D$ has exactly $k$ local branches, this 
is a fiber bundle with fiber a real $k$-torus. Its total space is where $M^+$ has a  codimension-$k$ 
corner. So the real codimension 1 strata (also called the \emph{walls}) of $M^+$  are indexed by the irreducible components of $D$. The  inclusion  $M\ssm D\subset  M^+$, is a homotopy equivalence, for the two have a  common deformation retract.  
\smallskip

So the normal crossing divisor $\Ds_\cub(L)$ in $\M^\st_\cub(L)$ gives a map 
\[
\M_\cub^+(L)\to \M^\st_\cub(L)
\]
and $\M_\cub^+(L)$ adds to $\M_\cub(L)$ a boundary  $\p\M_\cub(L)$ with corners. The $W$-action on $\M^\st_\cub(L)$ lifts 
to $\M_\cub^+(L)$.  By the above, it comes with a proper $W$-equivariant map 
$\p\M_\cub^+(L)\to \Ds_\cub(L)$ which over each stratum of complex codimension $k$ is a real 
$k$-torus bundle. 

Let $\p\M_\cub^+\subset \M_\cub^+$ be obtained from $\p\M_\cub^+(L)\subset \M_\cub^+(L)$ by passing to $W$-orbit  spaces. We  regard this as an embedding in the category of topological orbifolds. The main result of this section is:

\begin{theorem}\label{thm:nearthediscriminant}
The  $W$-equivariant  inclusion $\p\M_\cub^+(L)\subset \M_\cub^+(L)$ is $2$-con-nected and hence  the inclusions
\[
\p\M_\cub^+(L)\subset \M_\cub^+(L)\supset \M_\cub(L), \quad  \text{resp.}\quad  \p\M_\cub^+ \subset  \M_\cub^+\supset \M_\cub
\] 
induce isomorphisms on fundamental resp.\  orbifold  fundamental groups. 
\end{theorem}

The  rest of this section is devoted to the proof.

\subsection{Marked semistable cubic surfaces} The basic reference is \cite{naruki}.
A cubic surface is  semistable if and only if each of its singular points is either a node (an $A_1$-singularity) or a cusp 
(an $A_2$-singularity). There is only one strictly stable closed orbit and this is the orbit  represented by  the \emph{tricuspidal cubic surface} $X_\infty$ defined by $z_3^3=z_0z_1z_2$ (thus named because its singular set set  consists of  three cusps, namely  at $[1:0:0:0]$, $[0:1:0:0]$ and $[0:0:1:0]$). 
So the  compactification $\M^\sst_\cub$ of $\M^\st_\cub$ provided by geometric invariant theory is a normal projective one-point compactification of  $\M^\st_\cub$. We therefore denote the added point by $\infty$. We shall see that $\M^\sst_\cub$ is highly singular there.  

The following observation will be crucial for the proof of Theorem \ref{thm:nearthediscriminant}. 

\begin{lemma}\label{lemma:ample}
The  discriminant locus $\Ds_\cub^\sst:=\Ds_\cub\cup\{\infty\}$ supports  an ample divisor.
\end{lemma}
\begin{proof}
This follows   from the fact that  a positive multiple of the discriminant locus 
$\Ds_\cub^\sst:=\Ds_\cub\cup\{\infty\}$ can be defined by an  $\SL_4(\CC)$-invariant.
\end{proof}

We extend the $W$-cover $\M^\st_\cub(L)\to\M^\st_\cub$ to a $W$-cover 
 $\M^\sst_\cub(L)\to \M^\sst_\cub$ by taking the normalization of $\M^\sst_\cub$ in 
 $\M_\cub(L)$. A marked  tricuspidal cubic (which is by definition a point over $\infty$) then determines 
 a reflection subgroup of $W$ of type $3A_2$. For what follows, it is worth noting that such reflection subgroups make up a single conjugacy class in $W$ and that the 
 $W$-normalizer of such a reflection subgroup is able to permute the three factors in an arbitrary 
 manner and hence is isomorphic with
$\Sfrak_3\wr\Sfrak_3$. In particular, the $W$-orbit over $\infty$ has size $|W|/6^4=40$.
A construction due to  Naruki \cite{naruki} resolves its singularities in such a manner that the total transform of the discriminant locus  becomes a normal crossing  divisor. We recall this below.

\subsection{The Naruki compactification} The basic reference is \cite{naruki}.
This is a modification of the GIT compactification $\M^\sst_\cub$ 
obtained as a blowup of $\infty$ and so we focus on the the tricuspical surface $X_\infty$ defined by $z_3^3=z_0z_1z_2$.
First note that the algebraic 2-torus  $T_2$ consisting of  $t=(t_0, t_1, t_2)\in (\CC^\times)^3$ with 
$t_0t_1t_2=1$ appears as the identity component of the automorphism group of $X_\infty$: it takes  $[z_0:z_1 :z_2:z_3]$ to
$[t_0^{-1}z_0:t_1^{-1}z_1:t_2^{-1}z_2:z_3]$.  Indeed, $\aut(X_\infty)=T_2\rtimes \Sfrak_3$, where $\Sfrak_3$ permutes the last three coordinates. 
%The locus $\M^\sst_\cub$ is very singular at $\infty$.  We can improve the situation \`a la Kirwan by first blowing up the $\PGL_3(\CC)$-orbit of  $z_0z_1z_2=z_3^3$vand then do geometric invariant theory for the blowup. This amounts to the following:
 
 A transversal slice  at $z_3^3=z_0z_1z_2$ to  its $\PGL_4(\CC)$-orbit  in the projective space of cubic divisors on $\PP^3$ is given  by the  family
\[
\textstyle z_0z_1z_2 =z_3^3+\sum_{i=0}^2 (a_iz_i^2z_3+ b_iz_i^3), 
\]
where the parameters are $((a_0,b_0),(a_1,b_1), (a_2, b_2))\in (\CC^2)^3$.
This family deforms each of the three $A_2$-singularities independently. A cubic surface in this family is  singular precisely if for some $i\in \{0,1,2\}$ the polynomial $z^3+a_i z +b_i$ has a double root.
Note also that this slice is invariant under $T_2\rtimes \Sfrak_3$: the $\Sfrak_3$-action is obvious and   $t=(t_0, t_1, t_2)\in T_2$ takes  $a_i$ to $t_i^{2}a_i$ and $b_i$ to $t_i^{3}b_i$. The $T_2\rtimes \Sfrak_3$-orbit space of this slice is an affine variety.  We shall in 
fact see that it is  quasi-homogeneous cone, whose complex-analytic germ
at the  vertex is that of $\M^\sst_\cub$ at $\infty$.

When we consider the $L$-marked version we must not divide out by the  $\Sfrak_3$-action  and should  think of 
$a_i$ resp.\  $b_i$ as a symmetric function of degree $2$ resp.\  $3$ of auxiliary variables $c_i, c'_i, c''_i$ with zero sum $c_i+c'_i+c''_i=0$ (and with  $t\in T_2$ acting on these as scalar multiplication by $t_i$).
So we get a  singular cubic surface if and only if   for some $i$, the  $c_i, c'_i, c''_i$  are not distinct.

We can eliminate $c''_i$ by the substitution $c''_i=-c_i-c'_i$. Then the space of $T_2$-invariant polynomials 
in the remaining variables $c_0, c'_0,c_1, c'_1,c_2, c'_2$ has a  monomial basis, namely the monomials 
 for which the degree  in the pair $c_i,c'_i$ is the same for all $i$. 
This is the algebra of regular functions on an affine cone over $\PP^1\times \PP^1\times \PP^1$. 
The cone is defined by $\Ocal_{(\PP^1)^3}(-1,-1,-1)$, i.e., 
the exterior tensor product of the tautological line bundles $\Ocal_{\PP^1}(-1)$ of the three factors. 
More concretely, it is the cone over the Segre embedding 
\[
([c_0:c_1], [c'_0:c'_1], [c''_0:c''_1])\in \PP^1\times \PP^1\times \PP^1\mapsto [(c_ic'_jc''_k)_{i,j,k}]\in \PP^7.
\]
Here it is best to think of a factor $\PP^1$ as the projectivization of the standard $2$-dimensional 
reflection representation 
of $\Sfrak_3$, for this then renders evident an action of $\Sfrak_3\wr\Sfrak_3$ on this cone. 
The cone is resolved by a single blowup at its vertex; this gives  us  the total space of  the  line bundle above, so with exceptional divisor  
$\PP^1\times \PP^1\times\PP^1$. Each factor $\PP^1$ comes with three points, 
respectively defined by $c=c'$, $c'=c''$ and $c''=c$.
%defined by 
%$c=c'$, $c'=c''$, $c''=c$. 
Let us use an  affine coordinate on $\PP^1$ for which these become 
$0,1, \infty$. The  strict transform of the discriminant locus hits this 
exceptional divisor in the divisor defined by the preimages of these three points under a projection on a 
factor $\PP^1$ and  so 
has $3.3=9$ irreducible components crossing normally in $(\PP^1)^3$. So this makes the total transform of the discriminant  a  normal crossing divisor. 

This situation is reproduced  at each point of $\M^\sst_\cub(L)$ over $\infty$ and so this
gives us a model for the $W$-equivariant 
Naruki modification \cite{naruki}, denoted here by  
\[
\widehat\M_\cub(L) \to \M^\sst_\cub(L).
\]
with  $\widehat\M_\cub(L)$ smooth and such that the preimage  $\widehat\Ds_\cub(L)$ of $\Ds_\cub^\sst(L)$  is a normal crossing divisor. The exceptional set of the Naruki modification consists of $40$ copies of $(\PP^1)^3$; these are all irreducible components of $\widehat\Ds_\cub(L)$. 

We write  $\widehat\M_\cub\to \M^\sst_\cub$ for what we get after passing  to $W$-orbit spaces.

\begin{remark}
Allcock-Carlson-Toledo \cite{act:cubic_surfaces} describe $\M^\st_\cub$ as a ball quotient. 
In those terms, $\M^\sst_\cub$ appears as the Baily-Borel 
compactification
and $\widehat\M_\cub$ is its toroidal resolution (in the sense of 
Ash-Mumford-Rapoport-Tai \cite{amrt}), which for ball quotients is unique.
\end{remark}

Since $\widehat\Ds_\cub(L)$ is a normal crossing divisor in $\widehat\M_\cub(L)$, there is defined the real oriented 
blowup 
\[
(\widehat\M^+_\cub(L),\p\widehat\M^+_\cub(L))\to (\widehat\M_\cub(L), \widehat\Ds_\cub(L)).
\]
We recover   $(\M^+_\cub(L),\p\M^+_\cub(L))\to (\M_\cub(L), \Ds_\cub(L))$ if we remove from this map  everything that lies over $\infty$.

\begin{proof}[Proof of Theorem \ref{thm:nearthediscriminant}]
Since $\widehat\Ds_\cub(L)$ is the preimage of $\Ds_\cub^\sst$, it follows from 
Lemma \ref{lemma:ample}  that  $\widehat\Ds_\cub(L)$ supports a positive ample divisor.
Its complement in $\widehat\M_\cub(L)$ is smooth and so classical  Lefschetz (or Morse) theory tells us that if $U$ is an open  regular neighborhood of
$\widehat\Ds_\cub(L)$ in $\widehat\M_\cub(L)$ with smooth boundary (which we want to choose $W$-invariant), then the pair 
$(\widehat\M_\cub(L)\ssm U, \p U)$ (a complex manifold of complex dimension 4 with boundary)  is $3$-connected. If $U^+$ is the preimage of $U$ in $\widehat\M_\cub^+(L)$, then
$(\widehat\M_\cub(L)\ssm U, \p U)\cong (\widehat\M^+_\cub(L)\ssm U^+, \p U^+)$ 
     and the inclusions 
\[
(\widehat\M^+_\cub(L)\ssm U^+, \p U^+)\subset(\widehat\M^+_\cub(L), U^+)\supset (\widehat\M_\cub^+(L),\p\widehat\M_\cub^+(L))
\]
are $W$-equivariant   homotopy equivalences of pairs. So $(\widehat\M_\cub^+(L),\p\widehat\M_\cub^+(L))$  is also $3$-connected.
The latter   differs from $(\M^+_\cub(L), \Ds^+_\cub(L))$
by the part that lies over $\infty$, i.e., the closures of the  tricuspidal walls.  We take a closer look at what happens there.

The closure of such a wall is the total space of  circle bundle $E\to (\PP^1)^3$, where $E$ is the bundle of rays of $\Ocal_{(\PP^1)^3}(-1,-1,-1)$.
The intersection of the closure of $\p\M_\cub^+(L)$ with $E$ is the preimage $E_D$ of the divisor  $D\subset (\PP^1)^3$, where one of the affine  coordinates is $0$, $1$ or $\infty$. We give  $\PP^1$ a cell structure with $\{0, 1,\infty\}$ as its $0$-cells, the three real intervals between them as $1$-cells and the two hemispheres (upper and lower) as $2$-cells. It is then clear that
$(\PP^1)^3\ssm D$ is a union of cells of dimension $\ge 3$. If we  endow $E$ with  a cellular decomposition
such that the projection onto $\PP^1\times\PP^1\times\PP^1$ is a cellular map,  then  $E\ssm E_D$ is also  union of cells of dimension $\ge 3$. We thus find  that in  the triple 
\[
(\widehat\M^+_\cub(L), \p\widehat\M^+_\cub(L), \overline{\p\M^+_\cub(L)})
\]
the last two items make up a $2$-connected pair  (the difference is a union of cells of dimension $\ge 3$). We already saw that the first two items make up a $3$-connected pair and hence 
$(\widehat\M^+_\cub(L),\overline{\p\M^+_\cub(L)})$ is $2$-connected. 

The manifold with corners 
$\widehat\M^+_\cub(L)$ contains $\M^+_\cub(L)$ as an
    open-dense  union of strata  and hence the inclusion is a homotopy equivalence. The same is true of the inclusion $\p\M_\cub^+(L))\subset \overline{\p\M_\cub^+(L)}$
and hence 
\[
(\M^+_\cub(L), \p\M_\cub^+(L))\subset (\widehat\M^+_\cub(L),\overline{\p\M_\cub^+(L)})
\]
is a homotopy  equivalence. So $(\M^+_\cub(L), \p\M_\cub^+(L))$ is $2$-connected.
\end{proof}

\section{The discriminant and its normal bundle} 
One of the  goals of this section is to reconstruct the $W$-manifold $\p\M_\cub^+(L)$  entirely in terms of $W$, that is, without any reference to a cubic surface. This will lead to a determination of its orbifold fundamental group (and hence of $\M^\cub$) in terms of $W$ only.

\subsection{Degree 6 divisors on the Riemann sphere} 
In what follows we  use notation and results of the appendix. 

Let $C$ be a smooth projective curve of genus zero (so a copy of $\PP^1$). 
In the appendix we define for every  set $P$ of finite size $n\ge 3$ the space $\M_P$ of 
$\aut (C)$-orbits of injections  $\varphi: P\hookrightarrow C$. 
This is contained in the space $\M^\st_P$ of $\aut (C)$-orbits of stable maps  
$\varphi: P\to C$ (maps whose fibers have size $<n/2$).  
The latter is  projective if  $n$ is odd. If $n$ is even it is projectively completed 
by a finite set with each added point corresponding to a  partition of 
$P$ into two subsets of equal size $n/2$, each constituting a fiber of $\varphi$.  
The spaces $\M_P\subset \M^\st_P\subset \M^\sst_P$ come with an action of the permutation group 
$\Sfrak_P$ and yield  orbit spaces denoted   $\M_{[n]}\subset \M^\st_{[n]}\subset \M^\sst_{[n]}$.  
When $n\ge 5$, the irreducible components of the stable discriminant $\Ds^\st_P=\M^\st_P\ssm \M_P$ are in bijective correspondence with 2-element subsets of $P$, or equivalently, with the  transpositions in $\Sfrak_P$.
We recover the boundary divisor from the transposition as the one that is left pointwise fixed by it.

In case $C$ is defined over $\RR$, then  $\M_P$ and 
$\M^\st_P$  are also defined  over $\RR$.  There are two cases:  $C$ has a real point (the split case) or not (the nonsplit case). We shall always assume we are in the split case. Then stereographic projection from a real point shows that $C$ is $\RR$-isomorphic with $\PP^1$, so that in particular, $C(\RR)\cong \PP^1(\RR)$. A point of $\M_P^\st(\RR)$
has then 
the  interpretation of a $\PSL_2(\RR)$-orbit of  a stable maps $P\hookrightarrow \PP^1(\RR)$.  However, if  we divide out by $\Sfrak_P$, we get $\M_{[n]}$ resp.\ 
$\M^\st_{[n]}$: the space of positive  reduced resp.\ stable degree $n$ divisors on $\PP^1$ modulo projective equivalence. Then a real point  of $\M_{[n]}(\RR)$ is representable by an $n$-element subset of $\PP^1(\CC)$ invariant under complex conjugation. A connected component is of $\M_{[n]}(\RR)$ is then  given  by the integer $m$
which records
the number of complex conjugate pairs. Let us say that we are in the split component if $m=0$.

We  here concentrate on the case when  $P$ has size $6$. 
Then  $\M^\st_P$ is smooth and  parametrizes  $\aut (C)$-orbits of maps 
$P\to C$  whose fibers  have size at most 2. What makes this case special is that $\Ds^\st_P=\M^\st_P\ssm \M_P$ is then a normal crossing divisor  and that two irreducible components meet if and only if their associated transpositions pairwise commute. 
Let  
\[
\M^+_P\to \M^\st_P.
\]
be its   real oriented blowup.
The $\Sfrak_P$-action naturally lifts to $\M^+_P$ and the inclusion of its interior $\M_P\subset \M^+_P$ is a
$\Sfrak_P$-equivariant homotopy equivalence.

%The points of $\M^\sst_P\ssm  \M^\st_P$ are defined by the $\varphi$ having only two fibers, each of size 3 (so there are $\half\binom{6}{3}=10$ of those). While $\M^\st_P$ is smooth, this is no longer the case for  $\M^\sst_P$. 

\subsection{Normalization of the discriminant}
The main goal of this subsection is to identify (over $\RR$) the normalization of $\Ds_\cub$ with $\M^\st_{[6]}$.

Let $X$ be a stable cubic surface with a node at $o\in X$. If  we choose  affine 
coordinates $z_1,z_2,z_3$ for the ambient $\PP^3$ having $o$ as origin, then
the surface in question has an equation of the form 
\begin{equation}\label{eqn:form at node}
f(z_1,z_2,z_3)=f_2(z_1,z_2,z_3)+ f_3(z_1,z_2,z_3)
\end{equation}
where  $f_2$   defines  a  nondegenerate quadratic form and $f_3$ is homogeneous of degree $3$. The more intrinsic way of saying this is that the Zariski tangent space $T_oX$ of $X$ at $o$ is all of $T_o\PP^3$  and that
$f_2$ defines the tangent cone of $X$ in $T_o\PP^3$. So the  projectivization of this tangent cone 
 is  the conic $C\subset \PP(T_o\PP^3)$ defined by $f_2$. Note that  $f_3$ cuts out on 
$C$ a divisor $D$ of degree $6$ which parametrizes the lines on $X$ that pass through  $o$. A straightforward computation shows that $D$ cannot have  point of multiplicity 
three and so the  stability condition for $X$ translates into a stability condition for the divisor $D$: 
its  multiplicities are $\le 2$. 

The pair $(C,D)$ is a complete invariant of the pair $(X,o)$. Indeed, we  recover $(X, o)$ from $(C,D)$ by first embedding $C$ in its natural ambient projective  plane, 
namely, the one given by its anticanonical embedding  $C\hookrightarrow \PP_C:=\check\PP(H^0(C, \theta_C))$. 
We choose a total order $\prec$ on the six
points of $D$ (counted with multiplicity)
and  denote by $\bl_{D, \prec}(\PP_C)\to  \PP_C$ the blowup of
$D$ obtained in this order. This is well-defined, even if $D$ has multiplicities $>1$, for during this process  we  identify $C$ with its strict transform under a blowup. Then $X$  appears as the anticanonical image of $\bl_D(\PP_C)$ in a projective $3$-space (and so is independent of the total order order $\prec$). This anticanonical map contracts the smooth genus zero curves  with self-interaction $-2$ (these will be pairwise disjoint)
and thus $\bl_{D, \prec}(\PP_C)$ provides us with the minimal resolution of $X$. Among the contracted curves is (the strict transform of) $C$ in $\bl_D(\PP_C)$, as it   has self-intersection $4-6=-2$;  it is  contracted to produce the node $o\in X$. 

We have thus defined a surjective morphism 
$\M^\st_{[6]}\to  \Ds_\cub$ of degree one.  It is also finite (in case $X$ has more that one ordinary double point, it singles out one of them) and so this is a normalization.

If $X$ is defined over $\RR$ and $o$ is a real point of $X$, then we can take the form \eqref{eqn:form at node} to be defined over $\RR$ as well. Then its projectivized  tangent cone $C$  at $o$ is split if and only if $f_2$ is indefinite. We are then in the split component of $\M_{[6]}$ if and only if each line on $X$ through $o$ is defined over $\RR$.

\medskip
If $o$ is the only singular point of $X$ (which means that $D$ is reduced), then a total order $\prec$ on the points of $D$ 
%essentially 
amounts to giving an $L$-marking of $X$: number the points of $D$ by $1, 2, \dots , 6$ accordingly  and assign to 
$e_i\in L$ the line on $X$ through $o$ defined by it and assign  to $\ell$ the class of  a general line on $\PP^2_C$. This identifies 
$\M^\st_{\{1, \dots, 6\}}$ with an irreducible component of $\M^\st_\cub$. We record this as follows.

\begin{corollary}\label{cor:normalization}
The map $j:\M^\st_{[6]}\to  \Ds_\cub$ constructed above  is a normalization and defined over $\RR$. 
    An  $L$-marking of a stable  cubic surface with exactly one node gives rise to  a  labeling of the points of  $D$ by 
    a $W$-translate $P$ of $\{e_1,\dots,e_6\}$. 
    %a $6$-element set $P$.
    This normalization lifts to  an isomorphism  from $\M^\st_P$ onto an irreducible component of  $\Ds_\cub(L)$.

The latter isomorphism  is defined over $\RR$ if and only if this  irreducible component of  $\Ds_\cub(L)$ contains a point defined by a cubic surface~$X$
    with a node as its unique singular point through which pass six real lines on $X$.
\end{corollary} 

%The set $P$ that appears here is a $W$-translate of $\{e_1, \cdots, e_6\}$. 
We shall see in the next subsection how the $W$-translate $P$ of $\{e_1,\dots,e_6\}$
%this 
can be obtained  in a 
%more 
natural  manner.

In the appendix we introduce the (orbifold) line bundle $\Ls^\st_{[6]}$ over $\M_{[6]}$. It is defined by the functor which assigns to a degree $6$ divisor $D$ on $C$ the one-dimensional vector space  
$\Ls_D$  which via   Lemma \ref{lemma:cohomological interpretation} is identified with  
$H^1(C, \theta^2_C(-D))$. We take as here as the definition of $\Ls_D$. Lemma \ref{lemma:section} implies that this line bundle has  a smooth nowhere zero section (hence is topologically  trivial).

\begin{proposition}\label{prop:normal bundle identification}
Let $C$ be as above and  $D$ a stable degree 6 divisor  on $C$.
    Denote by $[(C,D)]$ the point of $M^\st_{[6]}$ it defines.
    Then:
\begin{enumerate}
\item[(i)] The normal space of  $j: \M^\st_{[6]}\to\M^\st_\cub$ 
at $[(C,D)]$ and $\Ls_D$ are canonically isomorphic. This identifies  the normal bundle of $j$ with $\Ls^\st_{[6]}$.
\item[(ii)] If $x\in C$ appears in $D$ with multiplicity $2$, then
$\Lcal_D$ is canonically isomorphic  with $T_xC$.
\end{enumerate}
\end{proposition}
\begin{proof}
We regard $C$ as anticanonically embedded in $\PP_C$ (i.e., by the degree 2 linear system $|\theta_C|$). Then
the normal bundle of $C$ is $\theta_C^{\otimes 2}$.

In view of the preceding, the normal of $j$ at $[(C,D)]$ is the space of infinitesimal deformations of $D$, regarded  as a positive $0$-cycle on $\PP_C$,  modulo those  that make $D$ lie  on an infinitesimal deformation of $C$.
The infinitesimal deformations of $D$ which lie on an infinitesimal deformation of $C$ can be
identified with the space of sections of the normal bundle of $C$ in $\PP^2_C$, so with $H^0(C, \theta_C^{\otimes 2})$. To see this: the conics in  $\PP^2_C$ form  a single orbit under its automorphism group  $\aut(\PP^2_C)\cong \PGL_3(\CC)$ and hence an infinitesimal deformation of $C$ is represented by an element of the Lie algebra of $\aut(\PP^2_C)$. The $\aut(\PP^2_C)$-stabilizer  of 
$C$ is  the image of  $\aut(C)\cong \aut(\PP^2_C)$ and the quotient of their Lie algebras can be
identified with $H^0(C, \theta_C^{\otimes 2})$. 

Hence  the normal line of $j$ at   $[(C,D)]$  is the cokernel of  the restriction map
\[
H^0(\theta_C^{\otimes 2})\to H^0(\Ocal_D\otimes_{\Ocal_C}\theta_C^{\otimes 2}).
\]
Since $\theta_C^{\otimes 2}$ has degree $4$,    $H^1(C, \theta_C^{\otimes 2})=0$ and so the long exact cohomology sequence of  the short exact sequence 
\[
0\to \theta_C^{\otimes 2}(-D)\to  \theta_C^{\otimes 2}\to \Ocal_D\otimes_{\Ocal_C}\theta_C^{\otimes 2}\to 0
\]
identifies
    this cokernel with $H^1(C, \theta_C^{\otimes 2}(-D))=\Ls_D$.

    For (ii), we write  $D=D'+(x)$ with  $x$ appearing in  $D'$ with multiplicity $1$. Then the fiber of $\theta^2_C(-D')$ over $x$ is naturally identified with $T_xC$ (in terms of a local coordinate $z$ at $x$, this assigns to 
$zf(z)\frac{d}{dz}\otimes \frac{d}{dz}$  the tangent vector $f(0)\frac{d}{dz}$) and so we have an exact sequence 
\[
0\to \theta^2_C(-D)\to  \theta^2_C(-D')\to i_{x*}T_xC\to 0, 
\]
where $i_x:\{x\}\subset C$. The degree of $\theta^2_C(-D')$ is $-1$ and so $H^1(C, \theta^2_C(-D'))=0$.
Hence the long exact cohomology sequence then yields an isomorphism $T_xC\cong H^0(C, \theta^2_C(-D))\cong \Ls_D$.
\end{proof}

\subsection{Reconstruction of the discriminant}
Recall that the irreducible components of $\Ds_\cub(L)$ are indexed by the reflections in $W$.  Here is how we associate to a reflection $s\in W$ in a $W$-equivariant manner  a set $P_s$ of size $6$ which via  Corollary \ref{cor:normalization} identifies $\M^\st_{P_s}$ with that  irreducible component.

Since  $W$ is a Coxeter group of type $E_6$, the $W$-centralizer $C_W(s)$ is a reflection subgroup of type $A_1\times A_5$ with the $A_1$ factor  represented by $s$. We realize this as a permutation group $\Sfrak_2\times \Sfrak_6$ by considering its action on the 27-element set  $\Escr\subset L$: the reflection $s\in W$ fixes $15$ elements of $\Escr$. The remaining 12, 
which we shall denote by $\Escr_s$,
form a `double six' in the sense that  $\Escr_s$  decomposes 
 uniquely
into two subsets 
$\Escr^+_s$  and $\Escr^-_s$  of size 6, each consisting of mutually
orthogonal
vectors, that  are exchanged  by $s$.
If $\alpha\in \Rscr$ is a root for $s$
 then we can take $\Escr^\pm_s$  to be 
the set of $e\in \Escr$ satisfying $\alpha\cdot e=\pm1$.
This decomposition
is preserved by  $C_W(s)$ and indeed, $C_W(s)$ is thus realized as
its stabilizer in the full symmetric group of  $\Escr_s$.  Here $s$ exchanges 
$\Escr^+_s$  and $\Escr^-_s$, and the other reflections  in
$C_W(s)$ act on each $\Escr^+_s$  and $\Escr^-_s$ as a transposition.
For example, if $\alpha=2\ell -e_1-e_2-\cdots -e_6$, then  
\[
\Escr^+_{s_\alpha}:=\{ e_i\}_{i=1}^6, \quad \Escr^-_{s_\alpha}:=\{\kappa-\ell +e_i\}_{i=1}^6.
\]
and the reflection in $e_i-e_j$ ($1\le i<j\le 6$) 
induces the evident transposition in either set.
We define $P_s$ as the orbit set  $\la s\ra \bs \Escr_s$.
This is a $6$-element set on which $C_W(s)$ acts
as the full permutation group with $s$ acting trivially and the other reflections in $C_W(s)$ acting as transpositions. We identify $P_s$ with the set of sums $e+s(e)\in L$, where $e$ varies over~$\Escr_s$. 

\medskip

\begin{lemma}\label{lemma:bijection}
Let $s$ and $s'$ be distinct reflections in $W$ which commute. Then $s'$ leaves four points of $P_s$ fixed and exchanges the remaining two (and similarly for $s$ acting on $P_{s'}$).  There is a bijection $g^s_{s'}$ between the fixed point set $s'$ in  $P_s$ and the fixed point set of $s'$ in $P_{s'}$, which is characterized by the property that $g^s_{s'}(a)$  is the unique such point not perpendicular to $a$.
\end{lemma}

\begin{proof}
Since the $A_1\times A_1$-subsystems of the $E_6$-Weyl group $W$ form a single conjugacy class it suffices to check this for one particular case, for example $s$ is the reflection in $\alpha:=2\ell-e_1-\cdots-e_6$ and $s'$ is the reflection in $\alpha':=e_1-e_2$. Then
\[
(P_{s})^{s'}\cong \{\kappa-\ell +2e_i\}_{i=3}^6=\{ \alpha  +2e_i\}_{i=3}^6
\]
and $s'$ acts as the transposition of  the first two elements $\alpha +2e_1, \alpha+2e_2$. Note that the second equality writes an element of $P_s$  as the sum of $\alpha$ and twice an element of $\Escr_s$.
Similarly, 
\[
(P_{s'})^s\cong\{2\ell-e_1-e_2-2e_i\}_{i=3}^6=\{\alpha'+2(\ell-e_1-e_i)\}_{i=3}^6
\]
We then note that  $(\alpha  +2e_i)\cdot (\alpha'+2(\ell-e_1-e_j))=4\delta_{ij}$. One checks that indeed  $g^s_{s'}(\alpha  +2e_i)=
\alpha'+2(\ell-e_1-e_i)$.
\end{proof}

In the notation of
the previous section, $\M_{P_s}$ is the moduli space of maps $P_s\to \PP^1$ that are injective,
whereas $\M_{P_s}^\st\supset \M_{P_s}$ allows the fibers of such a map to have  size $2$ (but not greater).  We simplify notation by writing $\M_s$ for $\M_{P_s}$ and $\M_s^\st$ for $\M^\st_{P_s}$.

We describe  the inverse map which identifies  an irreducible component 
of $\Ds_\cub(L)$ with some $\M^\st_s$.  For this we  represent an  element of $\Ds_\cub(L)$ by a marked cubic surface $(X, L\twoheadrightarrow H_2(X))$ and single out  a node $x\in X$ such that the associated  reflection in $W$ (see Subsection \ref{subsect:marked stable}) is $s$. Here is how this yields a point of $\M^\st_s$. According to that remark, $\Escr_s$ then maps to the classes of the lines on $X$ which pass through $o$  and this map factors through $P_s$. If $C_x$ denotes the projectivized tangent cone of $(X,x)$, then the resulting map $P_s\to C_x$ defines a point of $\M^\st_s$. Its image (as a zero cycle on $C_x$) is the divisor $D_x$ we defined in Subsection \ref{subsect:marked stable}.

We have described each component of $\Ds_\cub$ without reference to cubic surfaces.
Next we describe which subsets are identified under the map
$\coprod_s\M_s^\st\to\Ds_\cub(L)$, and how.  
This will yield a description of $\Ds_\cub(L)$, which still makes no
reference to cubic surfaces.  
Recall from Lemma \ref{lemma:linesonstablecubic} that two components of~$\Ds_\cub(L)$ meet just if
their associated reflections $s,s'$ commute.  
As we observed,  $s'$ then induces  a transposition  in $P_s$ and for the 
same reason  $s$ then induces  a transposition  in $P_{s'}$.  Since 
$s'$ acts as a transposition  in $P_s$, its fixed-point set in $\M_s^\st$ consists of the locus where two of the six points
of~$P_s$ map to the same point of~$C$, i.e.,
% determines 
a boundary divisor $\M^\st_{s,s'}$ of $\M_s^\st$. 
Likewise $s$ determines a boundary divisor 
$\M^\st_{s',s}$ of $\M_{s'}^\st$. We shall see that the two are naturally identified. 

\begin{lemma}\label{lemma:symmetry}
The subvariety $\M^\st_{s}\times_{\M_\cub(L)}\M^\st_{s'}$ of $\M^\st_{s}\times\M^\st_{s'}$ is the graph 
of an isomorphism $\M^\st_{s,s'}\cong\M^\st_{s',s}$. 
This isomorphism is induced by 
either of the two bijections $P_s\to P_{s'}$ which extend
the bijection $g^s_{s'}$ found in Lemma \ref{lemma:bijection}.
Such a bijection  lifts the  isomorphism 
$(C_x,D_x)\cong (C_{x'},D_{x'})$ of Lemma \ref{lemma:linesonstablecubic}.
 \end{lemma}
 
\begin{proof}
A point of $\M^\st_{s}\times_{\M_\cub(L)}\M^\st_{s'}$  is represented by a 
marked cubic surface $X$ with  two distinct nodes $x$, $x'$ 
such that $s,s'$ are their associated  reflections in $L$. To be precise, the point is  represented by
the pair 
\[
(P_s\to C_x, P_{s'}\to  C_{x'})
\]
The images of these  maps (as zero cycles) yields the divisors $D_x$ and $D_{x'}$.  Lemma \ref{lemma:linesonstablecubic} gives  an  isomorphism  $C_x\cong C_{x'}$ which takes $D_x$ to $D_{x'}$. 
It also shows how this isomorphism comes about:  to a line $\ell$ on $X$ through $x$ which is distinct from the line $\overline xx'$ we assign the unique  line through $x'$ which meets $\ell$. We  can express this in homological terms as follows.

If $\pi :\tilde X\to X$ is a minimal resolution, then the marking of $X$ factors over an isomorphism $L\cong H_2(\tilde X)$ (see \eqref{eqn:marking}). The class of the exceptional divisor over $x$ gives a root $\alpha$ with $s=s_\alpha$.  If $\ell$ is a line through $x$ which does not contain another node, and we identify $\ell$ with the class in $H_2(\tilde X)$ of its strict transform, then 
$2\ell +\alpha\in P_s$. Likewise for $x'$: we get a root $\alpha'$ such that $s'=s_{\alpha'}$ and if $\ell'$ is a line through $x'$ which does not contain another node, then $2\ell +\alpha'$ is the corresponding element of $P_{s'}$.
The intersection numbers  $\alpha\cdot\alpha'$, $\alpha\cdot \ell'$ and $\alpha'\cdot \ell$ are zero since the representative classes have disjoint support. So $(2\ell +\alpha)\cdot (2\ell' +\alpha')=4\ell\cdot\ell'$ and this is zero unless $\ell $ and $\ell'$ meet.

This proves that the isomorphism $D_x\cong D_{x'}$ is compatible with the bijection $g^s_{s'}$. It thereby proves the  lemma in case $x$ and $x'$ are the only nodes on $X$. The general case, which involves a minor modification of this proof, is left to the reader.
\end{proof}

Corollary \ref{cor:normalization} and Lemma \ref{lemma:symmetry} enable us to reconstruct the discriminant:  

\begin{proposition}%\label{prop:}
Let $\Ds(L)$ be the quotient of the  disjoint union $\coprod_s  \M^\st_s$, where $s$ runs over all the reflections in $W$, which is obtained  by identifying  $\M^\st_{s,s'}$ with $\M^\st_{s',s}$ via~$g_{s'}^{s}$ whenever $s$ and $s'$ commute.
Then there is a natural $W$-equivariant isomorphism $\Ds(L)\cong \Ds_\cub(L)$. In particular, each $ \M^\st_s$ appears as an irreducible component of $\Ds_\cub(L)$. \hfill $\square$ 
\end{proposition}

We shall denote the common image of $\M^\st_{s,s'}$ and $\M^\st_{s',s}$ in $\Ds_\cub(L)$ by $\M^\st_{\{s,s'\}}$.
More generally, 
if $S\subset W$ is a set a pairwise commuting reflections in $W$, then we put  $\M^\st_S:=\cap_{s\in S} \M^\st_s$ with  the intersection taken in $\Ds_\cub(L)$.  If $s\in S$, then $S\subset C_W(s)$ and  $P_{s}$  and hence each $s'\in S\ssm\{s\}$  determines a transposition in $P_s$. These will commute (hence there are at most three such, so that $S$ must have size $\le 4$) and  
define a  boundary stratum of $\M^\st_s$ which is identified with $\M^\st_S$.

\begin{remark}%\label{rem:}
We can obtain in the same manner all of the Naruki boundary $\overline \Ds_\cub(L)$ in terms of $W$ (by adding 40 copies of $(\PP^1)^3$ in a $W$-equivariant manner).  Our description of the Naruki resolution   also shows that
the  open embedding $\M^\st_{[6]}\hookrightarrow\Ds_\cub$ extends to a morphism from one resolution to another, namely from the Deligne-Mumford compactification  of  $\M_{[6]}$ to $\widehat  \M_\cub$.
\end{remark}

\subsection{Reconstruction of $\p\M_\cub^+(L)$}
We  shall see that Part (ii) of Proposition \ref{prop:normal bundle identification}  
enables us to construct a $W$-space $\N(L)$ and a 
$W$-equivariant map $\N(L)\to \Ds_\cub(L)$ which reproduces  $\p\M_\cub^+(L)\to \Ds_\cub(L)$. 

Let $s\in W$ be reflection.  Part (i) of Proposition \ref{prop:normal bundle identification} implies that  the normal bundle of the embedding $j_s: \M_s\hookrightarrow  \M_\cub(L)$
is naturally isomorphic to $\Ls^\st_{P_s}$ (which we will from now on  abbreviate by  $\Ls_s^\st$). 
The normal crossing divisor $\Ds_s$ in $\M_s^\st$ determines a  real oriented blowup 
$\M^+_s\to \M_s^\st$ whose total space $\M^+_s$ is a manifold with boundary having  $\M_s$ as its interior.
Let  $\Rs ay(\Ls_s^\st)\to \M^\st_s$ denote the (circle) bundle of rays in $\Ls_s^\st$ and write $\N_s\to \M^+_s$
for its pull-back along $\M^+_s\to \M_s^\st$, giving a fiber square
\[
\begin{tikzcd}
\N_s\rar\dar & \Rs ay(\Ls_s^\st)\dar  \\
\M^+_s\rar  &\M^\st_s
\end{tikzcd}
\]
Note that $\N_s$ is a manifold with corners. We shall use these $\N_s$  as building blocks to construct  a $W$-space  which yields a copy of $\p\M_\cub^+(L)$. 

Let   $s'$ be  a reflection in $C_W(s)\ssm \{s'\}$.
We claim that the three 2-torus bundles  over $\M^\st_{\{s,s'\}}$ defined by  the restrictions  to $\M^\st_{s,s'}$ of  the maps  $\M^+_\cub \to \M^\st_\cub$,  $\N_s\to \M^+_s$ and  $\N_{s'}\to \M^+_{s'}$  are canonically isomorphic. 
To see this, first note that since  $ \M^\st_s$ and  $\M^\st_{s'}$ intersect normally, the normal bundle of $\M^\st_{s'}$ restricted to $\M^\st_{\{s,s'\}}$ is the tangent bundle of $\M^\st_{s'}$ restricted 
to $\M^\st_{\{s,s'\}}$. The assertion then follows from   Proposition \ref{prop:normal bundle identification}-(ii), which implies 
that   the latter  is canonically isomorphic to the restriction of $\Ls^\st_s$ to $\M^\st_{\{s,s'\}}$. 

By identifying
the restrictions of $\N_s$ and $\N_{s'}$ to $\M^\st_{\{s,s'\}}$ for all such commuting pairs of reflections, we obtain a manifold with corners, which we shall denote by $\N (L)$. It comes with a $W$-action and we denote its $W$-orbit space by $\N$. 
This construction was devised for make the following to hold.

\begin{corollary}\label{cor:boundary with corners}
We have a natural  isomorphism of $\N(L)$ with $\p\M_\cub^+$ as manifolds with corners with $W$-action.

In particular, the fundamental group of $\N(L)$ can be  identified with that $\M_\cub(L)$ and 
the orbifold fundamental group of its $W$-orbit space 
$\N$ with  that of  the moduli space of cubic surfaces $\M_\cub$.  \hfill $\square$
\end{corollary}

\subsection{The homotopy type of a nerve}
If $S$ is a finite set of 
reflections in $W$, then  
\[
\N_S:=\cap_{s\in S} \N_s
\]
is nonempty if and only if the members of $S$ pairwise 
commute;  in that case this intersection is for every $s\in S$  a connected codimension $|S|$ corner of 
$\N_s^\st$. The  relative interiors $\mathring \N_S$  of  such $\N_S$  thus decompose $\N(L)$ into a  simple 
stratified space (into normally crossing submanifolds). 

When we regard $\{\N_s\}_s$ as a (closed) covering of $\N$, then its nerve  is the simplicial complex 
$K$ (with $W$-action) whose vertices  are the reflections in $W$ and for which a finite set $S$ of such reflections 
spans a simplex  $\sigma_S$ if and only the elements of $S$  pairwise commute. We already noted that the maximal size $S$ is $4$ and so 
$\dim K=3$. Two such commuting sets $S$, $S'$  or reflections of the same size $k\le 4$ form a single  $W$-orbit and hence $W$ 
acts transitively on the $(k-1)$-simplices. More is true: any two flags of simplices of the same  
termwise dimension are $W$-equivalent and so $W$ acts even transitively on the vertices of the same type of the barycentric subdivision $bK$ of $ K$.  This implies that any $3$-simplex of $|bK|$ is a strict fundamental domain of 
the $W$-action on $| K|$. 

Although this is not needed for the proof of our main theorem, we show that the  homotopy type of $ K$  is quite simple.
%In addition, the $W$-stabilizer of the corner $\mathring\N_S$ maps onto the full permutation group of $S$. In other words,
%$N_W(S)/C_W(S)=\Sfrak_S$.

\begin{proposition}\label{prop:bouquet}
The  homotopy type of $| K|$ is  a bouquet of $170$ two-spheres. 
\end{proposition}
\begin{proof}
Let $s_o$ be the reflection in the root $2\ell-(e_1+\cdots +e_6)$. All other reflections
in $W$ are indexed by subsets of $ \{1, \dots, 6\}$ of size $2$ or $3$: 
if $A$ has size $2$: $A=\{i,j\}$, then  $s_A$ is the reflection in $e_i-e_j$ and if $B$ has size $3$, then $s_B$ is the reflection in $\ell-\sum_{i\in B} e_i$.
The commutation relations can be expressed in these terms:
the reflections $s_A$ make up the reflections  of $W\ssm \{s_o\}$ which commute with $s_o$,  
two reflections $s_A$ and $s_{A'}$ commute if and only if $A\cap A'=\emptyset$, a reflection 
$s_A$ commutes with $s_B$ if and only if $A\subset B$ or $A\cap B=\emptyset$ and $s_B$ commutes with $s_{B'}$
if and only if $B\cap B'$ is a singleton.

We write $ K_2$ resp.\  $ K_3$  for the subcomplex of $ K$ spanned by the $s_A$'s resp.\ $s_B$'s.
Note that  $ K_2$ is the link of $s_o$  in $ K$  and hence the star of $s_o$ in $ K$  is the cone over $ K_2$. 
    Let $K^{(i)}$ be the subcomplex of~$K$ consisting of the simplices which
    have at most~$i$ vertices outside of~$\Star(s_o)$.  So
\[
    \Star(s_o)= K^{(0)}\subset  K^{(1)}\subset  K^{(2)}\subset  K^{(3)}\subset K^{(4)}= K.
\]
    We examine how each $|K^{(i+1)}|$ differs from its predecessor.
    Each component of $|K^{(i+1)}|-|K^{(i)}|$ contains the relative interior of a unique 
    $i$-simplex of~$K_3$, so may treat each of these $i$-simplices 
    separately.

    Passing from $|K^{(0)}|$
    to $|K^{(1)}|$ amounts to attaching, for each vertex $s_B$ of~$K_3$, the cone on
    the link of~$s_B$ in~$K^{(0)}$.  
This link  is the full subcomplex of~$K_2$ spanned by the $s_A$ with $A\subset B$ or 
$A\cap  B=\emptyset$, which is the join of two $3$-element sets. 
    Its homotopy type is a bouquet of four circles, and so adjoining the cone
    over it attaches $4$ two-cells (up to homotopy). 
    Since $K_3$ has $\binom{6}{3}=20$ vertices, 
    $|K^{(1)}|$ is homotopy equivalent to $|\Star(s_o)|$ with $4.20=80$ two-cells attached.

    Passing from $|K^{(1)}|$ to $|K^{(2)}|$ amounts to adjoining, for each
    edge $\overline{s_Bs_{B'}}$ in~$K_3$, the join of that edge with its link in~$K^{(0)}$.  
    One can check that the link is the edge
$\overline{s_{B\setminus B'}s_{B'\setminus B}}$ of~$K_2$.  The join is a tetrahedron,
    and its subcomplex lying in $K^{(1)}$ is the union of the two faces got by
    omitting $s_B$ resp.\ $s_{B'}$.  Therefore, attaching this tetrahedron does
    not change the homotopy type.  Repeating the argument for the 
    other edges of~$K_3$ shows that $|K^{(1)}|\to|K^{(2)}|$ is a homotopy equivalence.

    Finally we enlarge $|K^{(2)}|$ to $|K^{(4)}|$.  One can check that
    the triangles in $K_3$ are all $\Sfrak_6$-equivalent, for example 
    $\langle s_{123},s_{345},s_{561}\rangle$, and that each lies in a unique
    tetrahedron in~$K$, in this example got by adjoining $s_{246}$.  The example
    shows that the tetrahedron also lies in~$K_3$.  
    There are a total of~$30$ such tetrahedra.
    Therefore,
    $|K^{(4)}|$ is got from $|K^{(2)}|$ by attaching~$30$ tetrahedra along
    their $1$-skeleta.  
   Up to homotopy,
    this is equivalent to attaching the cones on these $1$-skeleta.
    Each $1$-skeleton has
    the homotopy type of a bouquet of $3$~circles, so  up to homotopy
    $|K^{(4)}|$ is got from $|K^{(2)}|$ by attaching $30.3=90$ two-cells.  
    Up to homotopy, we have built~$|K|$ by attaching $80+90=170$
    two-cells to the contractible subspace $|\Star(s_o)|$.
\end{proof}

\begin{remark}\label{rem:bouquet}
Proposition \ref{prop:bouquet}  implies that the $\ZZ W$-module $H_2( K)$ is the homology of the augmented simplicial chain complex of $\ZZ W$-modules
\[
0\to C_3( K)\to C_2( K)\to C_1( K)\to C_0( K)\to \ZZ\to 0
\]
(so with the trivial $\ZZ W$-module $\ZZ$ put in degree $-1$). The term $C_i( K)$ has as a $\ZZ W$-module a single generator: if $W_i\subset W$ is the stabilizer of an $i$-simplex, 
then $C_i( K)\cong \Ind^{W}_{W_i}\chi_i$,  where  $\chi$ is 
the sign character of the action of $W_i$ on the vertices of that simplex.
This allows us to determine $H_2( K)$ as a virtual $\ZZ W$-module.
\end{remark}

\section{The Clebsch surface and the Sylvester family}
In order to determine the  (orbifold)  fundamental group of $\M_\cub^+$, we need of course a base point. A smart choice is  a base point which comes with maximal symmetry. One such  is furnished by the point representing the Clebsch cubic, but this point does not lie on $\p\M_\cub^+$. The desire to reconcile these conflicting demands  motivates what follows.

\subsection{The Clebsch diagonal surface}\label{subsect:clebsch}
We fix  a 5-element set $\Vcal$ and write $\PP_\Vcal$ for the  4-dimensional projective space $\PP(\CC^\Vcal)$ and $\PP_\Vcal(\RR)$ for its real part. The Clebsch surface $X_C\subset \PP_\Vcal$ is defined by the two equations 
\[
\textstyle \quad  \sum_{i\in \Vcal} z_i=0; 
\quad \sum_{i\in \Vcal} z_i^3=0. 
\]
So it is a cubic surface in a hyperplane in $\PP_\Vcal$ and  comes with an evident action of the permutation group $\Sfrak_\Vcal$.  It is easily checked that $X_C$ is smooth. It is one of the  gems of 19th century algebraic geometry whose rich history  continues to this day.
Half a century ago, Hirzebruch \cite{hirz:clebsch} showed that it isomorphic with  a  Hilbert modular surface associated with  an order in $\QQ(\!\sqrt{5})$. We refer the reader  to this paper for some of its other interesting properties.
\smallskip

The  lattice $L$ we  fixed at the beginning is unique up to isomorphism. This gives us the freedom to make one particular choice. We use that liberty by stipulating that from now on, 
\[
L:= H_2(X_C).
\] 
So $W$ now appears as the group of orthogonal transformations of $H_2(X_C)$ which preserve the  anticanonical class of $X_C$.  Any nontrivial  automorphism  of a cubic surface 
acts nontrivially on its Picard group and so  the group $\Sfrak_\Vcal$ embeds in $W$. 
Our first business is  to determine  the image of this embedding and  the  lines on $X_C$.

We first note that the intersection of ${X_C}$ with a coordinate hyperplane $z_i=0$  consists of three  lines: for every splitting of $\Vcal\ssm \{i\}$ in two $2$-element subsets $\{J,K\}$ we get a  line 
defined by $(z_i=0, \sum_{j\in J} z_j=0, \sum_{k\in K} z_k=0$). So that gives us  a $\Sfrak_\Vcal$-orbit consisting of $5.3=15$ lines, 
all defined over $\RR$. 

A line on ${X_C}$ not of this form can be given  by a bijection $\varphi :\Vcal\cong \mu_5$ it is then the line $E_\varphi$ spanned by the point $x_\varphi\in {X_C}$ with homogeneous coordinates 
by $[\{\varphi(i)\}_{i\in \Vcal}]$ and its complex conjugate $\bar x_\varphi=x_{\varphi^4}$, given by $[\{\bar\varphi(i)\}_{i\in \Vcal}]$ (read $\varphi^k$ as the map $i\in \Vcal\mapsto \varphi(i)^k\in \mu_5$; this is a bijection when $k$ is not divisible by $5$).
Another bijection $\varphi'$ defines the same line  as $\varphi$ if  and only if $\varphi'$ is obtained from $\varphi$ or its complex conjugate $\bar\varphi$ by multiplication  with an element of $\mu_5$. So  this gives us 
$(5!/2)/5=12$ lines,  also  defined over $\RR$ and as $15+12=27$, there will be no others. These $12$ lines  form a single  $\Sfrak_\Vcal$-orbit  (which we shall denote $\Escr_C$), but decompose into two orbits for the alternating group $\Sfrak^+_\Vcal$. We shall denote these $\Sfrak^+_\Vcal$-orbits by  $\Escr^+_C$  and $\Escr^-_C$, although at this point  neither  of them is privileged. 
They make up a double six, whose $W$-stabilizer $\Sfrak_6\times\ZZ/2$
has central involution~$s_C$ sending each $E_\varphi=E_{\varphi^4}$ to
$E_{\varphi^2}=E_{\varphi^3}$.

Since the  $\Sfrak_\Vcal$-stabilizer of $\Escr^+_C$ is  $\Sfrak_\Vcal^+$, this means that if  we blow down the lines indexed by  $\Escr_C^+$, 
we get a projective plane with a $6$-element subset having $\Sfrak_\Vcal^+$ 
as its group of symmetries. 
There is only one such $6$-element subset of $\PP^2$  up to projective equivalence, namely the 6 points in 
the real projective plane $\PP^2(\RR)=\Pop(\RR^3)$ defined by the  lines though opposite vertices  of a  regular icosahedron $\Is \subset \RR^3$ centered at the origin. 
Let us identify ${X_C}$ with the blowup of $\PP^2$ in these six points (which are now labeled by the elements of $\Escr_C^+$), in other words, let us 
identify $\Escr_C^+$  with the set of antipodal vertex pairs of $\Is$ (or equivalently, the lines spanned by such pairs). So if $\ell\in \pic(S)$ denotes  the line class of $\PP^2$ pulled back to ${X_C}$, then $\{\ell\}\cup\Escr_C^+$ gives a basis of $L$.

The lines on $X_C$ can now be labeled nicely in terms of the combinatorial structure of the regular icosahedron $\Is$. 
An antipodal pair of edges  of $\Is$  defines a $2$-element subset of $\Escr_C^+$. Such a pair defines a reflection in $W$ (the reflection vector is the difference of the two exceptional  classes). 
This pair spans a plane in $\RR^3$  and determines a line in $\Pop(\RR^3)$.  The complexification of the strict transform of such a line yields  a line on ${X_C}$ and so this gives us 
15 lines on ${X_C}$. This is an $\aut({X_C})$-orbit of lines and hence these are the 15 lines obtained by hyperplane sections. 

This also enables us to recover $\Vcal$ in  terms of $\Is$: every antipodal pair of edges 
yields a line on ${X_C}$ which lies on a unique coordinate hyperplane section $z_i=0$ for some $i\in \Vcal$ and thus determines an element $i\in \Vcal$. Since $z_i=0$ meets $X_C$ in precisely 
 three such  lines, we thus have associated with $i\in\Vcal$ three antipodal pairs of edges of $\Is$. These $6$ edges lie on a unique cube which circumscribes $\Is$ (see the figure in \S 2 of \cite{hirz:clebsch}) and so this  identifies $\Vcal$ with the set of five circumscribed cubes of $\Is$.

An antipodal pair of  facets of $\Is$ defines a $3$-element subset  $\triangle\subset\Escr_C^+$. 
Such a  $3$-element subset defines a root in $\pic({X_C})$, namely $\alpha_\triangle:=\ell-\sum_{e\in \triangle} e$. 
So we have $20/2=10$ of those. The associated reflections $s_\triangle$ determine a set 
$\Ps$ of reflections in $H_2(X_C)$.
Two such reflections commute if and only if the associated facets have no edge in common. 
Hence every reflection in $\Ps$ fails to commute with precisely three 
of the other nine 
(and then generates with each of them a reflection subgroup of type $A_2$). It follows that 
the `Coxeter diagram' with vertex set  $\Ps$ is the Petersen graph. Note that this realizes $\aut({X_C})$ as the automorphism group of the graph $\Ps$. 

Since  no two of the three planes spanned by  the three edges of  $\triangle$ are perpendicular, $\triangle$ defines a 3-element subset of 
$\Vcal$. There are $\binom{5}{3}=10$ such and hence we also get an identification of $\Ps$
with the collection of 3-element subsets of $\Vcal$. If $I\subset \Vcal$ is a $3$-element subset, then
we denote by $s_I$ the reflection  in $W$ associated with $I$. 

Part of this discussion can be summed up as follows.

\begin{corollary}\label{cor:choice of six}
There are four $\Sfrak_\Vcal^+$-orbits in the set of reflections in $W$ (which has size 36): the singleton $\{s_C\}$, the $10$ reflections defined by the antipodal pairs of faces of regular icosahedron $\Is$,
their 10 conjugates   under $s_C$ and the  transpositions in $\Escr_C^+$ (this set has size $15$). 
The two $\Sfrak_\Vcal^+$-orbits of size $10$ make up a single  $\Sfrak_\Vcal$-orbit (giving us three 
$\Sfrak_\Vcal$-orbits).
\hfill $\square$
\end{corollary}

\begin{remark}
\label{rem:unique-S5}
    As the $\PGL_4(\C)$-stabilizer of~$X_C$, $\Sfrak_\Vcal$ is a subgroup of $\pi_1^\orb(\M_\cub)$.
    In fact it represents the unique conjugacy class of subgroups isomorphic
    to $\Sfrak_5$.  The key to this is 
    an invariant
    %a $\pi_1^\orb(\M_\cub)$-invariant
    Riemannian metric 
    %of negative sectional curvature 
    on
    the universal orbifold cover of
    $\M_\cub$, whose metric 
    completion is  negatively curved in the sense that it is a CAT($-1$) metric space
    \cite{A00}.  Such a space has the property that every finite group of isometries
    has a fixed point.  The points added in the completion lie over 
    $\Ds_\cub^\st$, and none of their stabilizers contains~$\Sfrak_5$.
    Therefore any $\Sfrak_5$ fixes a point which
    represents a smooth cubic surface admitting $\Sfrak_5$ symmetry.
    The only such surface is~$X_C$.
\end{remark}

 \subsection{The Sylvester family}
The Clebsch diagonal surface sits in what is called the \emph{Sylvester family} of cubic surfaces in  $\PP_\Vcal$. It is defined by 
\begin{equation}\label{eqn:sylvester}
\textstyle \quad  \sum_{i\in \Vcal} Z_i=0; 
\quad \sum_{i\in \Vcal} \lambda_i Z_i^3=0, 
\end{equation}
where the parameter $\lambda:=[(\lambda_i)_{i\in \Vcal}]$ is also an element of $\PP_\Vcal$,  the Clebsch surface being  given by $\tau:=[\{1\}_{i\in \Vcal}]$. It is thus called, because all its members have the Sylvester pentahedral form.  Note that the $\Sfrak_\Vcal$-action on  $\PP_\Vcal$ defined by  permutation the coordinates,  extends to the whole family 
(in the sense that it permutes the fibers).

\begin{remark}\label{rem:segre}
B.~Segre  (\cite{segre}, \S 85) proved that if  two smooth cubic surfaces appearing  in the Sylvester  family are isomorphic, then
they lie in the same $\Sfrak_\Vcal$-orbit.  He also determined the smooth cubic surfaces  not appearing in this family. 
\end{remark}

We now restrict the parameters of the  Sylvester family to the 
algebraic torus  $\Ts'\subset \PP_\Vcal $ defined by $\prod_{i\in\Vcal} \lambda_i\not=0$. 
A point $p\in X_\lambda$   is singular if and only if 
the differentials of $\sum_{i\in \Vcal} Z_i$ and   $\sum_{i\in \Vcal} \lambda_iZ^3_i$ 
are  proportional at $p$. This means that 
$\lambda_i p_i^2$ is independent of $i$. 
Assuming $\lambda\in \Ts'$, we can then take $\lambda_i p_i^2=1$ for all $i$, so that $p_i=\eps_i \lambda_i^{-1/2}$ for 
some sign $\eps_i\in\{\pm1\}$. 
The condition that $p$ lies on $X_\lambda$ then amounts to $\sum_{i\in \Vcal} \eps_i \lambda_i^{-1/2}=0$.

This suggests that for the torus $\Ts'$ we  make the base change  $t_i^2=\lambda_i^{-1}$. This  yields the algebraic torus $\Ts\subset \PP_\Vcal$ defined by $\prod_{i\in \Vcal}t_i \not=0$ and a regular covering $\Ts\to \Ts'$. If we  also make the coordinate change $Z_i=t_iz_i$, then the cubic surfaces of the above form are  now given by  
\begin{equation}\label{eqn:newfamily}
\textstyle X_t: \quad \sum_{i\in \Vcal} t_iz_i=0, \;\sum_{i\in \Vcal} t_i z_i^3=0 
\end{equation}
with  $(t_i)_{i\in \Vcal}\in \Ts$.  Note that the identity element of $\Ts$   corresponds to
the Clebsch surface.   The resulting family $\Xs\to \Ts$  still comes with an evident $\Sfrak_\Vcal$ action, but there is now more symmetry due to our base change:  if $\eps\in \mu_2^\Vcal$, then 
\[
([z], [t])\mapsto  ([(\eps_iz_i)_i], [(\eps_i t_i)_i])
\]
obviously leaves the family invariant.  We thus find that the group $\mu_2^\Vcal\rtimes\Sfrak_\Vcal$ acts on the family   $\Xs\to \Ts$.

Our justification for using $\Xs/\Ts$ as our model is in part explained by the following two lemmas.

\begin{lemma}\label{lemma:discriminant1}
The  cubic surface  $X_t$ defined the equations $\eqref{eqn:newfamily}$ with $t\in \Ts$  has no singularities except possibly at the $16$ points $\eps\in\mu_2^{\Vcal}$.  A singularity
exists there if and only if $\sum_{i\in\Vcal}\eps_it_i=0$, in which case it has nodal type ($A_1$).
\end{lemma}
\begin{proof}
Write $F$ and $G$ for the linear and cubic forms defining~$X_t$.  A singularity can only appear
where $dF$ and $dG$ are
proportional, which one argues as above holds only at $[(p_i)_{i\in\Vcal}]$ with
$p_i^2$ is independent of $i$.  By scaling
we may take all $p_i\in\mu_2$, proving our first claim.  This point lies on~$X_t$ if and only if
$\sum_i t_i p_i=0$, proving our second.  For the rest, by applying an element of
the deck group $\mu_2^\Vcal$, we may suppose all $p_i=1$ and 
that $\sum_{i\in\Vcal}t_i=0$. 

This singular point is an ordinary quadratic singularity if and only if the coefficients of $dF\wedge dG$ generate the maximal ideal of $\Ocal_{X_t, p}$. The coefficient of $dz_i\wedge dz_j$ in $dF\wedge dG$  is  $\pm 2t_it_j(z_i-z_j)(z_i+z_j)$ and so this is a unit times $z_i-z_j$. These linear forms indeed generate the maximal ideal of the point with all its  components equal.
\end{proof}

We shall refer to the hyperplane sections of $\Ts$ found in Lemma \ref{lemma:discriminant1} as  \emph{discriminant hyperplanes}. 
It is convenient to  label them in terms of subsets of $\Vcal$. They come as three  $\Sfrak_\Vcal$-orbits and 
each of these hyperplanes are defined over $\RR$: 
\begin{enumerate}
\item[$D_\emptyset$]  the hyperplane defined by $\sum_{j\in \Vcal} t_j=0$;
\item[$D_i=D_{\{i\}}$] here $i\in \Vcal$ and $D_i$ is  the hyperplane  defined by $t_i=\sum_{j\in \Vcal\ssm\{i\}} t_j$;
\item[$D_I$] where $I$ a $3$-element subset of $\Vcal$,  and $D_I$ is the hyperplane  defined by $\sum_{i\in I} t_i=\sum_{j\in \Vcal\ssm I} t_j$.
\end{enumerate}

Our focus will be on the last collection $\{D_I\}_I$, whose index set (the $3$-element subsets of $\Vcal$)  is by definition also the index set of  vertices of the Petersen graph $\Ps$. 
Let us denote by $D_\Ts\subset \Ts$ the union of the above divisors  (so this is the discriminant of the family $\Xs/\Ts$).

\begin{lemma}\label{lemma:discriminant2}
The discriminant $D_\Ts$  is  a simple normal crossing divisor in~$\Ts$, each of whose 
components is defined over $\RR$. When $t\in \Ts$, the singular set of the cubic surface $X_t$  has only nodal ($A_1$) singularities and as many as  the multiplicity of $D_\Ts$ at $t$; these cubic surfaces and their  singularities are all real when $t$ is. 

In particular, this family defines a morphism $\Ts\to \M_\cub^\st$.
\end{lemma}
\begin{proof}
Lemma \ref{lemma:discriminant1} shows that every singular point $p$ of $X_t$ determines  a discriminant  hyperplane section containing $t$ and that  this singular point is nodal. 
If $D_\Ts$ fails to be a normal crossing divisor at some $t$, then by applying an element
of $\mu_2^\Vcal$ we
can suppose $D_\emptyset$ passes through~$t$.  Then none of the $D_i$ pass through~$t$,
or else some $t_i$ would vanish.  A similar argument shows that $t$ cannot lie on three
divisors $D_I, D_{I'}, D_{I''}$ where $I\cap I'\cap I''$ has size~$2$.  The rest of the
analysis is left for the reader.
\end{proof}

\subsection{The Sylvester polytope}\label{}
We originally regarded $\PP_\Vcal$ as having homogeneous coordinates
$\lambda_i$ and containing $\Tcal'$ as a dense subset.  But our
alteration of the Clebsch family is easier to work with, so now
regard $\PP_\Vcal$ as having homogeneous coordinates $t_i$ and containing
$\Tcal$ as a dense subset.  The family
\eqref{eqn:newfamily}, a priori defined over~$\Tcal$, obviously
extends over~$\PP_\Vcal$.

\begin{definition}\label{def:s-polytope}
The \emph{Sylvester polytope} 
$\Pi\subset \PP_\Vcal(\RR)\iso \RR\PP^4$ is the closure of the connected component of the discriminant complement in $\RR\PP^4$ that contains
the  Clebsch point $[1{:}1{:}1{:}1{:}1]$. 
\end{definition}

\begin{proposition}\label{prop:about pie}
The Sylvester polytope $\Pi$ parametrizes real cubic surfaces. 
It is the projectivization of the cone in $\RR^\Vcal$ 
whose 10 facets are
defined by the supporting hyperplanes  
$\sum_{v\in I} t_v\le \sum_{j\in \Vcal\ssm I} t_j$, where $I\subset \Vcal$ is of size 3.  

Its set of vertices consists of two $\Sfrak_\Vcal$-orbits of size $5$: a  vertex of $\Pi$  has all but one  of its (homogeneous) $\Vcal$-coordinates equal to 1 and the remaining coordinate equal to either $0$ or $2$.

The first $\Sfrak_\Vcal$-orbit coincides with $\Pi^\infty:=\Pi\ssm \Ts$ and  each of its points  lies on six distinct facets. The rest $\Pi^\st:=\Pi\cap \Ts$ is a manifold with corners whose codimension $k$ corners parametrize (stable) cubic surfaces with precisely $k$ ordinary double points. 
\end{proposition}
\begin{proof}
It is clear that $\Pi$ is defined by the three types of inequalities $\sum_{j\in \Vcal} t_j\ge 0$,
$\sum_{j\in \Vcal\ssm\{i\}} t_j\ge t_i$, $\sum_{i\in I} t_i\ge \sum_{j\in \Vcal\ssm I} t_j$ (with $I\subset \Vcal$ of size $3$). The first two types are easily seen to be consequences of the last.  The verification that $\Pi_\infty$ is  as stated is straightforward.

Since $D_\Ts$ is a normal crossing divisor whose multiplicity $k$ locus  parametrize (stable) cubic surfaces with 
precisely $k$ ordinary double points, its follows that $\Pi^\st$ is a manifold with corners with the codimension $ k$ corners being as stated. \end{proof}

Since  $\Pi^\st$ parametrizes stable cubic surfaces, we have a natural continuous map $\Pi^\st\to \M_\cub^\st$. It   is clear that this map factors through  $\Sfrak_\Vcal\bs \Pi^\st\to \M_\cub^\st$. The preimage 
of the discriminant is the boundary, so that the interior $\mathring\Pi$ of $\Pi^\st$ is mapped to $ \M_\cub$.  
By B.~Segre's uniqueness theorem (see Remark \ref{rem:segre}), the  Sylvester form of a cubic surface is unique  up to permutation of its coordinates  and so the induced  map 
$\Sfrak_\Vcal\bs\mathring\Pi\to \M_\cub$ injective. Since the interior $\mathring\Pi$ of $\Pi$  is contractible, this endows all the surfaces parametrized by it  with an $L$-marking. In other words, this map lifts to a $\Sfrak_\Vcal$-equivariant map $\mathring \Pi\to \M_\cub (L)$. Proposition \ref{prop:about pie} implies that this lift 
extends (equivariantly) to a  map $\Pi^\st\to \M^\st_\cub (L)$.

\begin{remark}\label{rem:lifts}
One can show that a point of $\Pi^\infty$ represents  a cone over a smooth cubic  curve (which has a simply elliptic singularity of degree 3, hence is not semistable), yet the  map $\Pi^\st\to \M_\cub^\st(L)$ extends continuously to a map $\Pi\to \M_\cub^\sst(L)$ which takes  a point of $\Pi_\infty$ to  point representing a tricuspidal surface. Furthermore, 
if  $\Pi^+\to \Pi$ is the  strict transform of $\Pi$  under the (real) blowup  of  the vertex set $\Pi_\infty$, then 
$\Pi\to \M_\cub^\sst(L)$ lifts continuously  to the Naruki compactification: $\Pi^+\to \widehat\M_\cub(L)$.
\end{remark}

The definition of $\Pi$ makes it clear that its walls are in bijectively labeled by the $3$-element subsets $I\subset\Vcal$: a  wall of $\Pi^\st$ is an open subset of $D_I(\RR)$ for a unique such $I$. 
The monodromy (acting on $H_2(X_C)$)  along  as a simple loop around $D_I$ based at $\Pi^\circ$ will be  a reflection $s'_I\in W$. The resulting map $I\mapsto s'_I$ from the collection of $3$-element subsets of  $I\subset\Vcal$ to the set of   reflections in $W$  is clearly $\Sfrak_\Vcal$-equivariant.

\begin{corollary}\label{cor:good choice of 6}
For a  unique choice of a `six'  $\Escr_C^+$ in the double six $\Escr_C$,  the   monodromy  $s'_I$ with `base point' $\Pi^\circ$ around the facet indexed by the $3$-element subset $I\subset \Vcal$  is equal to the  reflection  $s_I$ we associated with $I$. So the map $\Pi^\st\to \M_\cub^+(L)$ has then the  property that the facet of $\Pi^\st$ defined by $I\subset \Vcal$ maps to $\N_{s_I}$.
\end{corollary}
\begin{proof}
Corollary \ref{cor:choice of six}  describes all the 
$\Sfrak_\Vcal$-orbits of reflections in $W$: besides the  singleton  $\{s_C\}$ we have a 15-element set 
and two 10-element sets. The latter two  are exchanged  by conjugation with $s_C$. One  them is indexed 
by pairs of antipodal facets of the regular icosahedron  $\Is$ and we also saw that these are in 
bijective correspondence with the  $3$-element subsets of  $\Vcal$; this defined the $\Sfrak_\Vcal$-equivariant map $I\mapsto s_I$. 

The latter map involved the  
choice of the six $\Escr_C^+$ in the double six $\Escr_C$. Since $s_C$ exchanges  $\Escr_C^+$ and $\Escr_C^-$, we can arrange that the image of the map $I\mapsto s'_I$ is this orbit. A little exercise shows that a $\Sfrak_\Vcal$-equivariant bijection between the $3$-element subsets of $\Vcal$ is necessarily  the identity and so then $s'_I=s_I$ for all $I$.
\end{proof}

For a more concrete computation the monodromy action of the ten generators on the 27 lines of $X_C$, see \cite{B26}.

From  now on we assume $\Escr_C^+$ has been chosen as in Corollary  \ref{cor:good choice of 6}. So the facets of 
$\Pi^\st$ are indexed by $\Ps$: with every $s\in \Ps$ is associated a facet $\Pi_s$ of $\Pi^\st$
which under the map $\p\Pi^\st\to \N(L)$ lands in the building block $\N_s$. Any face of $\Pi^\st$ of codimension $k\ge 1$ is a nonempty  intersection $\Pi_S:=\cap_{s\in S} \Pi_s$, where 
$S\subset \Ps$ is a pairwise commuting  set (in other words, the induced subgraph on $S$ consists of isolated points).
It  lands in $\N_S=\cap_{s\in S}\N_s$ with  its relative interior $\mathring \Pi_S$ mapping to $\mathring \N_S$. 
\smallskip

We observed that the lines on $X_C$ are all real and that the Sylvester polytope parametrizes 
real cubic surfaces. Hence $\mathring{\Pi}$ parametrizes real cubic surfaces all of whose $27$ lines are real. 
In view of the discussion in Subsection \ref{subsect:marked stable}, this is still the case on $\Pi^\st$. 
The following is then clear.

\begin{corollary}\label{cor:real component}
The relative interior $\mathring\Pi_s$ of $\Pi_s$ parametrizes real cubic surfaces with a single node through which pass $6$ real lines and is mapped by  $\Pi^\st\to \M_\cub(L)$ to the part of $\M_s(\RR)$ of $\Ds_\cub(L)$ which 
parametrizes injections  $P_s\hookrightarrow \PP^1(\RR)$ up to projective equivalence. \hfill $\square$
\end{corollary} 

\begin{remark}\label{rem:real component}
For a finite set $P$ of size  $\ge 3$, the space  injections $P\hookrightarrow \PP^1(\RR)$ given up to projective equivalence can be `rigidified' by 
choosing
distinct $p_0, p_1, p_\infty$ and  sending them to $0,1,\infty$. The remaining points of $P$ are then distributed over the three remaining  intervals and this makes it clear that every connected component of this space is contractible. So the facet $F$ in Corollary \ref{cor:real component} lands in a contractible connected  component of $\M_s(\RR)$. This serves as the  homotopy base point of 
$\M_s$ and $\M_{[6]}$  for the presentations of their fundamental groups that we find in Section \ref{sect:quantics} (see Remark \ref{rem:real base point}).
\end{remark}

%The map  $\Pi^\st\to \M^\st_\cub$ defined in Lemma \ref{lemma:lifts} has a  continuous lift $\Pi^\st\to \M^+_\cub (L)$. 

\begin{definition}\label{def:clebschext}
The  \emph{Clebsch extension} $\N^C(L)$  of $\N(L)$ is obtained by attaching copies of 
$\Pi^\st$: for  every $w\in W$ we attach $\Pi^\st$ to $\N(L)$  via the composite of $\p\Pi^\st\to \N$ with $w$, but identify two such attachments in the obvious manner if $w'w^{-1}\in \aut(\Ps)$.
\end{definition}

Note that $\N^C(L)\ssm\N(L)$ is a disjoint union  of copies of $\mathring\Pi$ (indexed by $W/\aut(\Ps)$ and so there are $2^4.3^3$ such). Since the boundary of $\Pi^\st$ is a topological $3$-sphere minus $5$ points, the homotopy effect of attaching a copy of $\Pi^\st$ is that of attaching $4$ three-cells (so we are attaching $2^6.3^3=1728$ three-cells in total).
It is clear that the $W$-action on $\N(L)$ extends to its Clebsch extension and that  $\N^C:=W\bs \N^C(L)$  is an extension of $\N$ with  $\N^C\ssm N= \aut(\Ps)\bs \mathring\Pi$. Attaching three-cells has no effect on the fundamental group and so 
$\N(L) \subset \N^C(L)$ induces an isomorphism of fundamental groups. We are now well-positioned to compute this fundamental group (which, as we have see,  is also the fundamental group of $\M_\cub(L)$).

\section{Determination of the fundamental group}

\subsection{The group $G(\Ps)$}
The Petersen graph $\Ps$ defines an \emph{Artin system}:  a group $\ar(\Ps)$ 
together with a generator  $\sigma_s$ for every vertex $s$ of  $\Ps$ such that the defining 
relations are prescribed by the graph (and are those satisfied by the reflections in its Tits representation), i.e.,  if $s$ and $s'$ are distinct vertices, then 
$\sigma_s$ and $\sigma_{s'}$ commute when $s$ and $s'$  are not connected  by an edge and  satisfy the braid relation $\sigma_s\sigma_{s'}\sigma_s=\sigma_{s'}\sigma_s\sigma_{s'}$  when they are.
The automorphism group of the graph $\aut(\Ps)=\aut(\Vcal)$ evidently acts on this presentation and so we can form the semidirect product $\ar(\Ps)\rtimes\aut(\Ps)$. 

Given a  vertex $s$ of 
$\Ps$,  the  full subgraph $\Ps_s$ of $\Ps$ spanned by the vertices not connected with 
$s$ is a hexagon and hence a Coxeter diagram of type $\hat A_5$. The stabilizer $\aut(\Ps)_s$ acts on this hexagon 
$\Ps_s$  with image  its full automorphism  group (a dihedral group of order 12) and trivial kernel, thus  identifying $\aut(\Ps)_s$ with 
$\aut(\Ps_s)$. It is a general fact that for any full subgraph $\G\subset\Ps$, the associated Artin group $\ar(\G)$ embeds in  $\ar(\Ps)$ \cite{lek:thesis} and so
$\ar(\Ps_s\cup\{s\})\rtimes \aut(\Ps_s)$ embeds in $\ar(\Ps)\rtimes\aut(\Ps)$. 

Since $\N_s$ is obtained from $\Rs ay(\Es_{[6]})$ (the quotient of $\Es_{[6]}$ by the action of the positive real scalars regarded as a subgroup of $\GG_m(\CC)$) by adding corners, it has the same  orbifold fundamental 
group as $\Es_{[6]}$. 
Theorem \ref{thm:central ext} of the appendix shows that this orbifold fundamental group 
is the  quotient of $\ar(\Ps_s\cup\{s\})\rtimes \aut(\Ps_s)$ defined by  the following \emph{rotation relations}
\begin{description} 
\item[(Rot$_s$)]  If $\{s_i\}_{i\in \ZZ/6}$  is a cyclic enumeration of the vertices of the hex\-a\-gon $\Ps_s$, then conjugation with $\sigma_s^{-1}\sigma_{s_1}\sigma_{s_2}\sigma_{s_3}\sigma_{s_4} \sigma_{s_5}$ takes
$\sigma_i$ to $\sigma_{i+1}$ for all $i\in \ZZ/6$.\end{description}
We shall denote this quotient group by $G(\Ps_s)$ and abuse notation a bit by continuing to denote the image of $\sigma_s$ in $G(\Ps_s)$ by the same symbol. The alternative description  Proposition \ref{prop:alternative presentation}
shows that $G(\Ps_s)$ is in fact a $\aut(\Ps_s)$-invariant quotient of $\ar(\Ps_s)$ on which $\aut(\Ps_s)$ acts  by  inner automorphisms.  Let us write $\theta$ for the  unique automorphism of $\Ps$ 
which  fixes $s$ and  rotates the associated hexagon by taking  $s_i$ to $s_{i+1}$ for all $i\in \ZZ/6$. 
This automorphism is order $6$. Such elements form in $\aut(\Ps)\cong\Sfrak_5$ a single conjugacy class (represented in  $\Sfrak_5$  by $(123)(45)$) and this conjugacy class generates $\aut(\Ps)$.
Since $\sigma_s^{-1}\sigma_{s_1}\sigma_{s_2}\sigma_{s_3}\sigma_{s_4} \sigma_{s_5}$ (as an element of $G(\Ps_s)$)  only depends on $\theta$, we shall denote it by $\rho_{\theta}\in G(\Ps_s)$. 

\begin{definition}%\label{def:}
The group $G(\Ps)$ is the quotient of $\ar(\Ps)\rtimes\aut(\Ps)$  obtained by imposing a rotation relation (Rot$_s$) for some vertex $s$ of  $\Ps$: conjugation with $\rho_\theta$ takes 
$\sigma_{s'}$ to $\sigma_{\theta(s')}$ for all $s'\in \Ps_s$.
\end{definition}

Since $\aut(\Ps)$ acts transitively on the vertices of $\Ps$, it does not matter which vertex we take, 
for these relations then will then hold for every vertex. Moreover,  $\rho_{\theta^{-1}}=\rho_\theta^{-1}$.

\begin{lemma}\label{lemma:}
Conjugation with $\rho_\theta $ takes 
$\sigma_{s'}$ to $\sigma_{\theta(s')}$ for \emph{every} vertex $s'$ of $\Ps$.
\end{lemma}
\begin{proof}
There are only three vertices of $\Ps$ not in $\{s\}\cup\Ps_s$. We index them by $\ZZ/3$ as follows:
for every $i\in \ZZ/3$, $t_i$ is the unique vertex  of $\Ps$ which is connected with $s$, $s_i$ and $s_{3+i}$. So we must show that conjugation with the rotation $\rho_\theta $ takes $\sigma_{t_i}$ to $\sigma_{t_{i+1}}$ or equivalently, that 
$\rho_\theta  \sigma_{t_i}=\sigma_{t_{i+1}}\rho_\theta$. Let us write this out for $i=0$: this means that we want that 
\[
\sigma_s^{-1}\sigma_{s_1}\sigma_{s_2}\sigma_{s_3}\sigma_{s_4} \sigma_{s_5}\sigma_{t_0}\stackrel{?}{=}
\sigma_{t_1}\sigma_s^{-1}\sigma_{s_1}\sigma_{s_2}\sigma_{s_3}\sigma_{s_4} \sigma_{s_5}.
\]
In view of the commutation relations among the generators appearing here,  this amounts to
\[
\sigma_{s_1}\sigma_{s_2}\sigma_{s_3}\sigma_{t_0}\sigma_s\stackrel{?}{=}\sigma_s \sigma_{t_1}\sigma_{s_1}\sigma_{s_2}\sigma_{s_3}.
\]
Since  $t_1,s_1, s_2,s_3, t_0, s$ make up a hexagon, this is  a valid identity 
implied by (Rot$_{s_5}$), as both define the same rotation fixing $s_5$.
\end{proof}

\begin{corollary}\label{cor:artin quotient}
The group $G(\Ps)$ is a $\aut(\Ps)$-invariant quotient of $\ar(\Ps)$  characterized by the property that conjugation with $\rho_\theta$ permutes the 
Artin generators according to $\theta$. In particular $\aut(\Ps)$ acts on this quotient  by inner automorphisms. 
\end{corollary}
\begin{proof}
This follows from the fact that $\aut(\Ps)\cong \Sfrak_\Vcal$ is generated by the conjugacy class of $\theta$.
\end{proof}

The Artin generators $\sigma_s$ belong to a single conjugacy class of $\Art(\Ps)$ which generates $\Art(\Ps)$.
In view of Corollary \ref{cor:artin quotient} the image of this conjugacy class in $G(\Ps)$ has the same property. We shall refer to its elements  as \emph{simple elements} of $G(\Ps)$.

\subsection{Maps associated with the Sylvester polytope}
We return to the Clebsch extension defined in the previous section.
The space  $\Pi^\st$ is contractible and so we may regard the map 
\[
\diamond: \Pi^\st\to \N^C(L)
\]
associated with the identity as the base point of $\N^C(L)$. We use this base point also for some of its (sub)quotients. We write $W\diamond$ for the $W$-orbit of $\diamond$. Note that there is a subtlety here: any 
$g\in\aut(\Ps)\ssm \{1\}$ changes $\diamond$ (but not its image) and so 
the  $W$-orbit  of $\diamond$ is a \emph{free} orbit. 

The inclusion $\N(L)\subset \N^C(L)$ defines an isomorphism on fundamental groups. So
the  fundamental group $\pi_1(\N^C(L),\diamond)$ can be identified with the 
fundamental group of $\M_\cub(L)$. Although $\N^C$ is not an orbifold, its  orbifold fundamental group
 can now be unambiguously defined (and 
identified with the orbifold fundamental group of $\M_\cub$) as 
the fundamental groupoid  of $\N^C(L)$ restricted to $W\diamond$, denoted  $\Pi(\N^C(L),W\diamond)$, but endowed with the $W$-action. This is indeed a group:
for  $\g, \g'\in \Pi(\N^C,W\diamond)$, the product $\g'\g$ is the path composite of $w(\g')$ and $\g$, where 
$w\in W$ is such that $\g$ has end point $w(\diamond)$ (recall that we use the categorical convention for path composition).

For $s\in \Ps$, the orbifold fundamental group of $\N_s^+$ can also be described in these terms:
the path should then stay in the $C_W(s)$-orbit of $(\N_s^+, \diamond)$, denoted
\[
\N^C(C_W(s)),C_W(s)\diamond)
\]
  The orbifold fundamental group of $\N_s^+$ is then $\Pi(\N^C(C_W(s)),C_W(s)\diamond)$ endowed with its $C_W(s)$-action. As we will see below, the advantage of this description is that the geometry now smoothly translates into group theory. 
\\  
 
In order to make the connection with the presentations found in Section \ref{sect:quantics} we shall use (often  tacitly) Corollary \ref{cor:real component} and its subsequent Remark \ref{rem:real component}.

Recall that for  any reflection $s$ of $W$,  the building block $\N_s$ comes with  a circle action. The reflection $s$ induces the  order 2 element  of that action.  Assume now that   $s\in \Ps$. Then the  path from $\diamond$ to $s(\diamond)$ which follows half a circle orbit in $\N_s$ in the  positive direction to its opposite defines an element of $\Pi(\N^C(L), W\diamond)$ that  we  shall denote by $\hat\sigma_s$. Observe  that this semicircle can be taken to lie  in $\N_S$ for any set $S$ of pairwise commuting reflections in $W$ which contains $s$ and so 
$\hat\sigma_s$ has a representative of that type. This implies that if $s'$ is also in $\Ps$, then $\hat\sigma_s$  and $\hat\sigma_{s'}$ commute and of from $\diamond$ to $ss'(\diamond)=s's(\diamond)$.
The element $\sigma_s$ and the six elements  $\hat\sigma_{s'}$ indexed by the hexagon $\Ps_s$ thus obtained  generate the orbifold fundamental group of $C_W(s)\bs \N_s$ (in the groupoid incarnation above). We have the braid relation  if $s'$ and $s''$ are connected in  $\Ps_s$, the
 $\hat\sigma_{s'}\hat\sigma_{s''}\hat\sigma_{s'}=\hat\sigma_{s''}\hat\sigma_{s'}\hat\sigma_{s''}$ and homotopic paths from $\diamond$ to $s's''s'(\diamond)=s''s's''(\diamond)$. So the assignment $\sigma_{s'}\mapsto \hat\sigma_{s'}$ (with $s'\in \Ps_s\cup\{s\}$) produces an isomorphism 
\[
 \Phi_s: G(\Ps_s)\cong \Pi(\N^C(C_W(s)),C_W(s)\diamond)\cong \pi^\orb_1(C_W(s)\bs \N_s).
 \] 
By the very definition of $G(\Ps)$, this  implies that the  same assignment (but with now $s'\in \Ps$) gives a homomorphism 
$\Phi: G(\Ps)\to  \Pi(\N^C(L),W\diamond)\cong \pi^\orb_1(\N)\cong \pi^\orb_1(\M_\cub)$.

\begin{theorem}\label{thm:main}
The homomorphism 
\[
\Phi: G(\Ps)\to  \Pi(\N^C(L),W\diamond)\cong \pi^\orb_1(\M_\cub)
\]
is an isomorphism.
\end{theorem}
\begin{proof}
Since $\Phi_s$ is an isomorphism, every path which stays in the $C_W(s)$-orbit of
$(\N_s,\diamond)$ is in the image of $\Phi_s$. This $C_W(s)$-orbit  contains all the $W$-translates of the $\N_s$ which have with $\N_s$ a wall in common. In other words   all the `wall crossings' of $\N_s$ are realized  by elements of $G(\Ps_s)$.
Since any element of $\Pi(\N^C(L),W\diamond)$ is an obtained by a sequence of  wall crossings (a \emph{gallery} of $W$-translates of $\N_s$), it follows that 
$\Phi$ is onto.

We next check that $\Phi$ is injective. This follows if we show that every relation in $ \Pi(\N^C(L),W\diamond)$ is a product of relations, with each factor realized in the image of $G(\Ps_s)\to G(\Ps)$ or one of its $G(\Ps)$-conjugates. For this it suffices to observe
that every elementary homotopy   
$\Pi(\N^C(L),W\diamond)$ is reproduced in a $G(\Ps)$-conjugate of  $G(\Ps_s)$. By an elementary homotopy we mean one that is defined by a codimension 2 face of $\N(L)$. Such a codimension 2 face is $W$-conjugate to a codimension face of the form $\N_s\cap \N_{s'}$ for some  distinct $s,s'\in \Ps$ and hence takes indeed place in $G(\Ps)$-conjugate of  $G(\Ps_s)$ (the relation being just a commutativity relation).
\end{proof}

\subsection{Monodromy quotients of $G(\Ps)$} 
It will be convenient to have the following notion at our disposal.

\begin{definition}%\label{def:}
For a positive  integer $\ell$, the \emph{$\ell$-truncation}  of $G(\Ps)$ is the quotient $G(\Ps)\twoheadrightarrow G(\Ps)_\ell$ obtained by dividing out by the normal subgroup by the $\ell$th power of a simple generator. 
\end{definition}

In view of the interpretation  $G(\Ps)$  as the orbifold fundamental  group of $\M_\cub$ it follows that this quotient defines a connected unramified orbifold  covering $\M_\cub[\ell]\to \M_\cub$ with covering group $G(\Ps)_\ell$. Since $\Ds_\cub\subset \M^\st_\cub$ has only normal crossings (in the orbifold sense), this covering extends  (in the complex analytic orbifold setting) as a ramified  $G(\Ps)_\ell$-covering:
\[
\begin{tikzcd}
\M_\cub [\ell]\dar\arrow[r, hook]  & \M^\st_\cub[\ell]\dar\\
\M_\cub\arrow[r, hook] & \M^\st_\cub 
\end{tikzcd}
\]
with simple cyclic order $\ell$ ramification along $\Ds_\cub$. We shall identify these coverings  for $\ell=2$ and  
$\ell=6$ 
and identify the truncations $G(\Ps)_2$ and $G(\Ps)_6$  as monodromy groups.

\begin{proposition}\label{prop:squares of generators}
The map which takes a simple element $\sigma_{s}$  of $G(\Ps)$ to $s$ induces an isomorphism $G(\Ps)_2\cong W$ and this is part of an equivariant isomorphism $\M^\st_\cub[2]\cong \M^\st(L)$ of $\M^\st_\cub$-coverings.
In particular, $G(\Ps)_2$ appears as the image of a  monodromy representation. 
\end{proposition}
\begin{proof}
Since $\M_\cub (L)$ is connected and  $\M_\cub (L)\to \M_\cub$ is  unramified as an orbifold covering with covering group $W$, we must have a surjection $G(\Ps_s)\to W$. It is clear that this maps a simple element to a reflection and hence factors through a surjection $G(\Ps_s)_2\to W$. By elementary covering theory,  the covering 
$\M_\cub [\ell]\to \M_\cub$ lifts to an unramified  orbifold  covering $\M_\cub [\ell]\to \M_\cub (L)$.
Since both source an target  have order two ramification along $\Ds_\cub$, it follows that this extends to an unramified covering $\M^\st_\cub [\ell]\to \M^\st_\cub (L)$.  It remains to show that this  covering is trivial.
For this it suffices to observe  $\M^\st_\cub (L)$ is simply connected. This is well-known, but for completeness we give  a quick proof.

As  in  Subsection \ref{subsect:marked stable} we identify $\M_\cub (L)$ with the set of pairs $(p_5,p_6)\in \PP^2\times\PP^2$ which together  with the  standard projective basis elements (the vertices of a tetrahedron) form a $6$-tuple in general position. 
This is a Zariski open subset of the form  $\PP^2\times\PP^2\ssm D$,   where $D$ is a divisor of  $\PP^2\times\PP^2$.  We then observe that the complement of  the singular part $D_\sing$ of $D$  
parametrizes  stable marked cubic surfaces and is open-dense in $\M^\st_\cub(L)$. 
Since $D_\sing$ is of complex codimension $\ge 2$ in $\PP^2\times\PP^2$, its  
complement
is simply connected. It then follows that $\M^\st_\cub(L)$ is also simply connected.
\end{proof}

We next show that something similar is true for the pair $(\M^\st_\cub[6],G(\Ps)_6)$.
 This comes from the  construction  which assigns to a smooth cubic surface $X\subset \PP^3$ the 
$\mu_3$-cover  $Y_X\to \PP^3$  which is totally ramified along $X$. In terms
of equations: if $X\subset \PP^3$ is defined by the homogeneous  cubic form  $F\in \CC[z_0,z_1, z_2,z_3]$, then
$F-z_4^3$ is a defining equation for $Y_X$. So $Y_X$ is a smooth cubic $3$-fold with 
$\mu_3$-action which 
has $X$ as its fixed point set. The  group $H_3(Y_X)$ is free abelian of rank $10$ and the  intersection pairing makes it a unimodular symplectic lattice. It comes with an action of $\mu_3$ which preserves the intersection pairing. Strictly speaking, the notation  $Y_X$ is a bit misleading as $Y_X$ is determined by $X$ only up to a covering transformation (i.e., an element of $\mu_3$).

Let us abbreviate  $\omega:=e^{2\pi\sqrt{-1}/3}$ and denote by  $\Ecal$ for the Eisenstein ring $\ZZ[\omega]/(\omega^2+\omega+1)$ (as embedded in $\CC$). We write  $\Ecal^\pm$ when $\Ecal$ is endowed with the hermitian form $(x,y)\in \Ecal\mapsto \pm x\overline y$. %The symplectic $\Ecal$-module $H_3(Y_X)$ is isomorphic with the symplectic sum of one  copy of $\Ecal$ and four copies of $(\Ecal)'$, where the symplectic pairing on $\Ecal$ is given by its orientation and the one on $(\Ecal)'$ by the opposite orientation. 
So $\Lambda:=H_3(Y_{X_C})$ is a $\Ecal$-module. We  turn  this into a hermitian $\Ecal$-module by the form
\[
h(a, b):= \frac{(a\cdot (\omega b))\omega +  (a\cdot (\overline\omega b))\overline \omega}{\omega-\overline \omega}\in \Ecal.
\]
According to \cite{act:cubic_surfaces}, this makes $\Lambda$ isomorphic with  $\Ecal^-\perp 4\Ecal^+$. So if we extend $h$  to a nondegenerate hermitian form $h_\CC$ on $\Lambda_\CC:=\CC\otimes_{\Ecal} H$, then it  has signature $(4,1)$. The  open subset $\BB\subset \PP(\Lambda_\CC)$ defined by the vectors $v\in \Lambda_\CC$ with $h(v,v)<0$ is a copy of the  complex unit $4$-ball. It is an orbit for the  action of   $\PU(\Lambda_\CC)$ on $\PP(\Lambda_\CC)$ and $\PU(\Lambda_\CC)$ acts properly on it. In fact, $\BB$ is  a symmetric domain for 
$\PU(\Lambda_\CC)$.  

Because of the aforementioned ambiguity, the monodromy group $\G$ of this family takes its values in $C_{\Sp(\Lambda)}(\mu_3)/\mu_3=\U(\Lambda)/\mu_3$ (rather than $\U (\Lambda)$). Allcock-Carlson-Toledo 
\cite[Thm.\ 2.14]{act:cubic_surfaces} prove that $\G$ is of index $2$ in $\U(\Lambda)/\mu_3$ and that the composite map 
\[
\G\subset \U(\Lambda)/\mu_3\to \U(\Lambda)/\mu_6=\PU(\Lambda)
\]
is an isomorphism. So this realizes $\G$ as an arithmetic  group faithfully acting on $\BB$.
It is not hard to check that a simple loop in $\M_\cub$ around a smooth point of $\Ds_\cub$ is mapped to a hexaflection
(a pseudoreflection of order $6$). Such a hexaflection has as a fixed point set  in $\BB$ a hyperball. 
Since $\G$ is the image of a monodromy factorization,  we have a group epimorphism $G(\Ps)\twoheadrightarrow \G$. The simple elements of $G(\Ps)$ map to  hexaflections,  and so this epimorphism  factors though an   epimorphism $G(\Ps)_6\to \G$.

In \cite{act:cubic_surfaces}  the following is proved:  these  hyperballs (fixed point sets of hexaflections) cross normally,  their union $\Ds_\BB$ is locally finite and we have an  isomorphism of varieties
\[
\G\bs(\BB, \Ds_\BB)\cong (\M^\st_\cub,\Ds_\cub).
\]

\begin{proposition}\label{prop:cubes of generators}
This map defines  an isomorphism $G(\Ps)_6\cong \G$ 
and this is part of an equivariant isomorphism 
$(\M^\st_\cub[6],G(\Ps)_6)\cong (\BB, \G)$
of $\M^\st_\cub$-coverings.
In particular, $G(\Ps)_6$
appears as the image of a  monodromy representation. 
\end{proposition}
\begin{proof}
This uses the same argument as the proof of its companion Proposition \ref{prop:squares of generators}: 
The epimorphism  $G(\Ps)\twoheadrightarrow G(\Ps)_6$
defines a connected  orbifold  covering 
of $\BB\ssm \Ds_\BB$. Our assumption implies that  this covering extends across  every smooth point  of $\Ds_\BB$ without ramification. Since $\Ds_\BB$ has normal crossings, this is even true everywhere and we thus obtain a  covering of $\BB$ without ramification. But $\BB$ is contractible and so this covering must be trivial and hence 
$G(\Ps)_6\to \G$ is an isomorphism.
\end{proof}

\begin{remark}
    Baldi, Farb, Javanpeykar and Stover \cite{BFJS} showed that the
    kernel of $\pi_1(\M_\cub)\to\Gamma$ is a characteristic subgroup, and used this
    to show that complex conjugation is the only nontrivial automorphism of $\M_\cub$
    as a complex-analytic orbifold.
\end{remark}

We next address the intermediate case $G(\Ps)_3$. This  seems  not to be covered by the literature.

The ring  $\Ecal/2\Ecal$  may be identified with  the Galois field $\FF_4$  with  the elements of $\FF_4$ being represented by $\{0,1,\omega, \omega^2=\overline\omega\}$. This takes complex conjugation in  $\Ecal$ to the  Frobenius automorphism of $ \FF_4$ (which we shall denote by $x\mapsto x'$).  We regard $\FF_4$  as endowed with the hermitian  form $(x,y)\in \FF_4\mapsto xy'$ (note that  $xx'\in \FF_2$).

We choose an hermitian $\Ecal$-lattice isomorphism  $\Lambda\cong \Ecal^-\perp 4\Ecal^+$. 
This determines  an isomorphism of  $\Lambda_2:=\Lambda/2\Lambda$ with  $5\FF_4$ as hermitian $\FF_4$-vector spaces. 
By \cite[Thm. 7.22]{act:cubic_surfaces} the group $\U(\Lambda)$ is generated by $\{\pm I\}$ and the hexaflections in the `short roots':
a short root means a vector $v\in \Lambda$ with $h(v,v)=1$, and its hexaflection scales it by $-\omega^2$ and is the identity on $v^\perp$.
Since $(-\omega^2)^3=-1\equiv1$, each hexaflection reduces to a triflection in an anisotropic element of $\Lambda_2$.  
One can check that every anisotropic element is the image of a short root. Also, their triflections are known to generate  
$\U(\Lambda_2)\iso \U_5(\FF_4)$.  
Therefore $\PU(\Lambda)\to\G_2:=\PU(\Lambda_2)$ is surjective.
The latter group has order  $2^{10}.3^5. 5.11$, got by counting the number of 
orthonormal bases of 
$5\FF_4$ and dividing by~$3$.
%This group 
%coincides with $\G_2:=\PU(\Lambda_2)$ because $\dim_{\FF_4}\Lambda_2=5$ and  $|\FF_4^\times|=3$ are relatively prime.
It also happens to be simple with the confusing ATLAS name ``$\U_5(2)$''.

Since the surjection 
$G(\Ps)_6\cong\PU(\Lambda)\to\PU(\Lambda_2)$ sends the Artin generators to order~$3$ elements,  
it factors through an epimorphism $f:G(\Ps)_3\to\PU(\Lambda_2)$.
Consider the corresponding
%The epimorphism \eqref{eqn:mod2} defines an 
intermediate Galois cover
\begin{equation}\label{eqn:covers}
  (\M^\st_\cub[3], \Ds_\cub[3])\xrightarrow{\pi}(\M^\st_\cub[\G_2], \Ds_\cub[\G_2])\to (\M^\st_\cub, \Ds_\cub)  
\end{equation}
where $\pi$  is the formation of the orbit space for the action of $\ker (f)$ on the first pair. The second map is a $\G_2$-covering which has order three ramification over $\Ds_\cub$. Since the composite is universal for order $3$ ramification along $\Ds_\cub$, the map $\pi$ is unramified and is in fact a universal cover of $\M^\st_\cub[\G_2]$.
Hence the following are equivalent:

\begin{proposition}\label{prop:G3}
The following  three equivalent properties hold:
\begin{enumerate}\label{enum:eq}
    \item [(i)] $f$ is an isomorphism,
    \item[(ii)] $\pi$ is an isomorphism,
    \item[(iii)] $\M^\st_\cub[\G_2]$ is simply connected.
\end{enumerate}
\end{proposition}

We have not succeeded in finding a conceptual proof.

\begin{proof}
Let $G_{25}$ be the $25$-th exceptional group in Shephard-Todd's list of finite complex reflection groups. It is known that 
$G_{25}$ has order $648$ and 
is isomorphic to the 
quotient of four strand braid group
by the relation
where a braid generator is forced to
have order $3$.
%G_{25} is the reflection group of the Eisenstein $E_6$-lattice 
Consider the subgroup $K$ of $G(\Ps)_3$ generated by an $A_3$ subdiagram inside $\Ps$.
%for example by the images of $\sigma_{4 5},  \sigma_{ 1 3}, \sigma_{2 5}$.
Then $K$ is a quotient of $G_{25}$. 
Coset enumeration\footnote{
The code, available via the link
\url{https://tinyurl.com/5fe77ymf}, runs in the browser at  
\texttt{sagecell.sagemath.org} and takes a few seconds.
The specific presentation of $G(\Ps)_3$
used has 14 generators and 54 relations.
There are four standard generators 
$s_1, \dotsb, s_4$ of the symmetric group
$\Aut(\Ps) = \Sfrak_5$ and ten generators 
$\lbrace \sigma_v \colon v \in \Ps \rbrace$
of $\Art(\Ps)$. There are ten standard 
Coxeter relations defining $\Sfrak_5$, 
forty relations of the form
$s_i \sigma_v s_i^{-1} = \sigma_{s_i v}$ 
specifying the conjugation action of 
$\Sfrak_5$ on $\Art(\Ps)$, 
one relation of the form 
$\sigma_v \sigma_{v'} \sigma_v = \sigma_{v'} \sigma_v \sigma_{v'}$ 
for an edge $\lbrace v, v' \rbrace$
of $\Ps$, one relation of the form 
$\sigma_v \sigma_{v'} = \sigma_{v'} \sigma_v$
for a non-edge $\lbrace v, v' \rbrace$ 
of $\Ps$, one relation of the form \eqref{EqRelation}
and one relation of the form $\sigma_v^3 = 1$.}
on a computer gives 
$\abs{G(\Ps)_3 / K} = 21120$. Therefore
$ \abs{G(\Ps)_3}  \leq 21120 \cdot 648 = 2^{10}\cdot 3^{5} \cdot 5 \cdot 11
= \PU(\Lambda_2)$.
This shows that $f$ is an isomorphism. Its equivalence with the other two properties is clear. 
\end{proof}

\smallskip

We do not know whether such an interpretation exists for any other pair  $(\M^\st_\cub[\ell], G(\Ps)_\ell)$.

\section{Appendix: orbifold fundamental groups involving $\M_{[n]}$}\label{sect:quantics}

\subsection{Positive divisors on $\PP^1$}  We first fix some notation:

\begin{notation}
For a space $Z$ and an integer $n\ge 0$, we denote by  $\Cs_n(Z)\subset Z^n$ the configuration space of $n$ numbered distinct points of $Z$ and by  $\Ss_n(Z):=\Sfrak_n\bs \Cs_n(Z)\subset \Sfrak_n\bs Z^n$ the space of $n$-element subsets of $Z$. 
Note that the group of topological automorphisms of $Z$ acts on the these spaces. We shall use this mostly in a complex algebraic setting. 
\end{notation}

Let $C$ be a curve isomorphic to $\PP^1$. Choose a complex vector space $V$ of dimension 2 and identify 
$C$ with the projective line $\PP(V)$. This  vector space is unique up to a scalar multiplication: 
if $V'$ is any other such vector space, then we have a compatible isomorphism $V\cong V'$ which is unique 
up to a scalar. Note that $\Aut(C)=\PGL(V)$. For an  integer $n\ge 0$, we write $\CC[V]_n$  for the space    of 
homogeneous polynomials on $V$ of degree $n$. This  is a vector space of dimension $n+1$ which 
we shall regard as a  representation of $\GL(V)$ (so the  scalar $\lambda\in \GL(V)$ acts 
on $\CC[V]_n$ as multiplication with $\lambda^{-n}$). Its projectivization $\PP(\CC[V]_n)$ is the space of degree $n$ divisors on 
$C$ and is thus identified with the $n$th symmetric power $\sym^n C$.

We now assume $n\ge 3$. The $\SL(V)$-stable locus  $\CC[V]_n^\st$ consists of the degree $n$ forms  whose divisor in $C$  has all its multiplicities $<n/2$. It defines the open subset $(\sym^n C)^\st\subset \sym^n C$.
It  is clear that $\CC[V]_n^\st$ contains the locus of square free elements (denoted  $\CC[V]_n^\red$) and so $ (\sym^n C)^\st$ contains $\Ss_n(C)$.

The orbit space  $\aut(C)\bs\Ss_n(C)$ is  independent of the choice of $V$  and we bring this out in the notation by denoting it $\M_{[n]}$. Indeed, this can be understood as the moduli space of pairs consisting of a smooth projective genus zero curve and an $n$-element subset of it. We write $\M^\st_{[n]}:=\PGL(V)\bs \PP(\CC[V]_n^\st)$ accordingly and  
denote by $\Ds_{[n]}\subset\M_{[n]}^\st$ the locus pa\-ram\-e\-triz\-ing singular stable divisors. It is a hypersurface (the \emph{discriminant}),  given by an invariant of degree $n(n-1)$ and its  complement $\M_{[n]}$ is an affine orbifold.

When $n$ is odd, $\M^\st_{[n]}$ is projective and when $n$ is even, it admits a projective  (and normal) one-point compactification $\M^\sst_{[n]}$ with the added point representing a divisor  supported by two points, each with multiplicity $n/2$.

 These  varieties admit natural $\Sfrak_n$-covers. Concretely, given  a finite set $P$ of size $n\ge 3$, let $\M_P$ be  the $\aut(C)$-orbit space of injective maps 
 $\varphi: P\hookrightarrow C$ (this is just a copy of the moduli space usually denoted $\Mcal_{0,n}$). 
 It is  a smooth affine variety of dimension $n-3$ on which the permutation group $\Sfrak_P$ acts  with orbit variety $\M_{[n]}$. 
The normalization of $\M^\st_{[n]}$ in $\M_P$ gives us the $\Sfrak_P$-equivariant extension $\M_P\subset \M^\st_P$.
A point $\M^\st_P$  is given by an $\aut(C)$-orbit of maps $\varphi:P\to C$ whose  fibers have  size $<n/2$. 
The discriminant $\Ds_P:=\M^\st_P\ssm  \M_P$ has now several  connected components. These are labeled by 
the  transpositions in $\Sfrak_P$: the associated component is  the divisorial  part of its fixed point set of in $\M^\st_P$.
\smallskip

\subsection{The mapping class group of an $n$-punctured sphere} 

We choose distinct points $o, \infty$ on $C$  and 
put $C^\circ:=C\ssm \{o, \infty\}$. Let $T\subset \PSL(V)$ denote the subgroup which stabilizes each of these points. The action of $T$ on the tangent space of $C$ at $o$ identifies $T$ with $\GG_m$, but it is convenient to distinguish the two. Note that $T$ acts  on $C^\circ$  simply transitively. 

Consider the configuration space  $\Ss_n(C^\circ)$  of $n$-element subsets of $C^\circ$.
The action of  $T$ on $C^\circ$  induces one on  $\Ss_n(C^\circ)$ and we write $\M_n(C^\circ)$  for its orbit space. As we shall recall below, it has a familiar interpretation, which  will give us a presentation of its orbifold fundamental group.

Let $T\hookrightarrow T^n$ be the diagonal embedding. The choice of a base point of $C^\circ$ gives an isomorphism $C^\circ\cong T$ and this identifies  $\M_n(C^\circ)$  with the $\Sfrak_n\times T$-orbit space of  $\Cs_n(T)\subset T^n$ 
(so the space $n$-tuples of distinct points of $T$ given up to a translation). We prefer however  to work with the kernel of the product
character $T^n\to T$  (which we shall denote by $T_{n-1}$). Note that the  homomorphism   $T_{n-1}\subset T^n\to T\bs T^n$ is an isogeny: it is surjective with kernel  the diagonally embedded copy of $\mu_n$ in $T_{n-1}$.  
We put  
\[
T_{n-1}^\reg:=T_{n-1}\cap \Cs_n(T).
\]
Thus   $\M_n(C^\circ)$ now appears as  the $\Sfrak_n\times\mu_n$-orbit space  of $T_{n-1}^\reg$ and also as the  $\mu_n$-orbit space  of $\Sfrak_n\bs T_{n-1}^\reg$.
It is clear  that $\Sfrak_n\bs T_{n-1}^\reg\subset \Ss_n(T)$ is the set of $n$-element subsets of $T$ whose product is $1$.
One such subset is $\mu_n\subset T$. When we consider it as an element of $\Sfrak_n\bs T_{n-1}^\reg$, we denote it by $[\mu_n]$ and  this will be our base point. We use the same notation for its image  in  one of the above quotients.
We often identify $\mu_n$ with $\ZZ/n$ by taking $\zeta_n:=\exp(2\pi\sqrt{-1}/n)$ as its generator.

The action of $\Sfrak_n$ on  $T_{n-1}^\reg$ is free 
and it is well-known \cite{nguy} that the fundamental group 
$\pi_1(\Sfrak_n\bs T^\reg_{n-1}, [\mu_n])$ is the affine Artin group $\ar(\hat A_{n-1})$ of type $\hat A_{n-1}$. 
Its generators are indexed by  $\ZZ/n$: for  $i \in \ZZ/n$, the generator $\sigma_i$  makes  
$\zeta_n^i$ and $\zeta_n^{i-1}$  trade places (where $\zeta_n:=\exp(2\pi\sqrt{-1}/n)$), 
while keeping the other   powers of $\zeta_n$ fixed: $\zeta_{n}^{i}$ and $\zeta_{n}^{i-1}$  
first travel to each other by staying on the unit circle,  
then they pass each other by `keeping on the right' (traversing  a small arc) until they meet the unit 
circle again  then continue to their journey  to the former position of the other. This can be done by staying in $T^\reg_{n-1}$. There is in fact a natural choice:  note that $\g (t):= (1-t)+t\zeta_n^{-1}$  traverses the interval $[1,\zeta_n^{-1}]$, so its interior  stays inside  the unit circle and hence $1/\g(t)$ traverses an arc from $1$ to $\zeta_n$ whose  interior stays outside the unit circle. Consider the loop 
\[
\tilde\g: t\in [0,1]\mapsto \{ \g(t)^{-1}, \g(t)\zeta_n,\zeta_n^2, \dots, \zeta_n^{n-1}\}\in \Sfrak_n\bs T^n.
\]
Since the $n$ elements on the right hand side stay distinct and have  product  $1$, this is in fact a loop in $\Sfrak_n\bs T^\reg_{n-1}$.
We then  let  $\sigma_k$ be represented by the $\zeta_n^k$-translate of $\tilde\g$. So  the $\mu_n$-action permutes these generators: $\zeta_n$ takes $\sigma_i$ to  
$\sigma_{i+1}$.
The defining (Artin)  relations are:  $\sigma_i$ and $\sigma_j$  commute unless $j=i\pm 1$, in which case we have the simple braid 
relation $\sigma_i\sigma_j\sigma_i=\sigma_j\sigma_i\sigma_j$. Since  
$\Sfrak_n\bs T^\reg_{n-1}\to  \M_n(C^\circ)$ is a $\mu_n$-cover,  we thus have  produced an isomorphism 
\begin{equation}\label{eqn:artinextended}
\pi_1^\orb(\M_n(C^\circ), [\mu_n])\cong \ar(\hat A_{n-1})\rtimes \mu_{n}, 
\end{equation}
where the generator $\zeta_n\in\mu_n$ is represented by the loop 
$\theta \in [0, 2\pi/n]\mapsto \exp(\sqrt{-1}\theta)[\mu_n]$.
We shall derive from this a presentation for $\pi_1^\orb(\M_{[n]}, [\mu_n])$.

The inclusion $C^\circ\subset C$ induces an inclusion $\Ss_n(C^\circ)\subset \Ss_n(C)$. The  inclusion of groups $T\subset\PGL(V)$ induces subsequently a morphism  $\M_n(C^\circ)\to \M_{[n]}$. The latter is in fact a fiber bundle in the orbifold sense, a typical fiber being $\Cs_2(\PP^1\ssm \mu_n)$. Since that fiber is connected,  the  map $\M_n(C^\circ)\to \M_{[n]}$ is a surjection on fundamental  groups. We first determine the fundamental  group of $\M_n(C^\circ)$  and then determine the relations we need to impose to get the fundamental  group of $\M_{[n]}$. 

If $\M_n(C)$ denotes the $T$-orbit space of $\Ss_n(C)$, then the map  $\M_n(C^\circ)\to \M_{[n]}$  factors as $\M_n(C^\circ)\subset\M_n(C)\to \M_{[n]}$. Our justification for introducing this, is the following lemma.

\begin{lemma}\label{lemma:pi1diagram}
The morphism $\M_n(C)\to\M_{[n]}$
induces an isomorphism on orbifold fundamental groups.
\end{lemma}
\begin{proof}
Consider the following commutative diagram of fundamental 
groups with exact rows (base points are omitted)
\[
\begin{tikzcd}
\pi_1(T)\rar\dar & \pi_1(\Ss_n(C)) \rar\arrow[d, equal] & \pi_1^\orb(\M_n(\PP(V))) \rar\dar  & 1\\
\pi_1(\aut(C))\rar & \pi_1(\Ss_n(C)) \rar & \pi_1^\orb(\M_{[n]}) \rar & 1
\end{tikzcd}
\]
The lemma then follows from the fact that the first vertical arrow is surjective (of any infinite cyclic group on to a group of order two). 
\end{proof}

The orbifold fundamental group of $\M_{[n]}$ can be identified with the mapping class group of the pair 
$(\PP^1, \mu_n)$, that is,  the connected component group of the group of orientation preserving 
diffeomorphisms of $\PP^1$ which fix $\mu_n$ as a subset. 
A presentation of this group is  due to J.~Birman (see Remark \ref{rem:presentation1}) and as we will 
explain, the following presentation is essentially equivalent to hers. Since we want the presentation to be compatible with the mapping class group interpretation, our path composition is here the functorial one 
and hence words in a fundamental group must  be read from right to left.

\begin{corollary}\label{cor:presentation1}
The orbifold fundamental group of $\M_{[n]}$ is the 
quotient  of $\ar(\hat A_{n-1})\rtimes \mu_n$ by the relations $\zeta_n \equiv \sigma_1\sigma_2\cdots \sigma_{n-1}$
and $\zeta_n^{-1}\equiv \sigma_{n-1}\sigma_{n-2}\cdots \sigma_1$.
\end{corollary}
\begin{proof}
Observe that  $\M_n(C^\circ)$ is Zariski open in $\M_n(C)$ and that its complement $\M_n(C)\ssm \M_n(C^\circ)$ is a divisor which has two irreducible components:
one is represented by the elements in $\M_n(C)$ which have a point in $\infty$ and the other by those having a point in $o$. So the inclusion $\M_n(C^\circ)\subset \M_n(C)$ is a surjection on fundamental groups with the kernel generated by two elements representing simple loops  $\delta_\infty, \delta_o$ in $\M_n(C^\circ)$ around such a 
component. 

One readily checks that the   element $\zeta_n\sigma_{n-1}\sigma_{n-2}\cdots \sigma_1$ represents the braid in $C^\circ$ which keeps all elements of $\mu_n\ssm \{1\}$ almost fixed 
(\footnote{But if we want to keep the product $1$, we must allow them to move a little.})  
and makes  $1$ follow the loop $\alpha_R$  which first traverses the interval $[1,R]$ (with $R\gg 1$), then the circle $|z|=R$ fully in the counterclockwise direction and then returns to its original position along $[R,1]$. 
We remind the reader of our convention that composition in fundamental groups is written right-to-left.  So $\zeta_n\sigma_{n-1}\sigma_{n-2}\cdots \sigma_1$ means the 
braid represented by $\sigma_1$, then the one represented by $\sigma_2$, etc., ending with $\zeta_n$.
So this represents a simple loop $\delta_\infty$ as above. The loop 
$\zeta_n^{-1}\sigma_{1}\sigma_{2}\cdots \sigma_{n-1}$ has a similar interpretation, but with $1$ traversing the loop $1/\alpha_R$ which first traverses 
$[1,1/R]$, then $|z|=1/R$ in a clockwise direction and finally $[1/R,1]$ and so  this represents a loop  $\delta_0$.
\end{proof}  

\begin{remark}\label{rem:presentation1}
The added relations in the presentation in Corollary \ref{cor:presentation1} broke the $\mu_n$-symmetry.
It is however clear that the relations imply their cyclic transforms: $\zeta_n \equiv \sigma_{k+1}\sigma_2\cdots \sigma_{n-1}\sigma_0\cdots \sigma_{k-1}$ for every $k\in \ZZ/d$ and likewise for $\zeta_n^{-1}$.
We can use this to eliminate $\mu_n$ from our generating set;  the relations then say that $\sigma_1\sigma_2\cdots \sigma_{n-1}$ is cyclically invariant, has  $\sigma_{n-1}\sigma_{n-2}\cdots \sigma_1$ as its inverse,  and is of order $n$,  so that we get in fact a quotient of $\ar(\hat A_{n-1})$.
We can subsequently eliminate $\sigma_0$ as to get even a quotient of the ordinary braid group  $\ar(A_{n-1})$. This gives us Birman's presentation (\cite{birman}, Thm.\  4.5). For many applications it is however convenient to include these redundant generators as this exhibits more of the symmetry (see also Proposition \ref{prop:braidpres} below).
\end{remark}

Let us denote the quotient of $\ar(\hat A_{n-1})\rtimes \mu_n$ defined by Corollary \ref{cor:presentation1} by
$\pi_{[n]}$ so that the corollary  identifies the orbifold fundamental group of $\M_{[n]}$ with this group.
We have a homomorphism  $\pi_{[n]}\to \Sfrak_{\ZZ/n}$ which assigns to $\sigma_i$ the transposition 
$s_i:=(i-1, i)$ and to $\zeta_n$ the $n$-cycle $i\mapsto i+1$: indeed, the latter is equal to $s_1s_2\cdots s_{n-1}$.  This homomorphism is clearly onto. The associated unramified $\Sfrak_{\ZZ/n}$-covering  
 is $\M_{\ZZ/n}\to \M_{[n]}$ and hence the kernel of $\pi_{[n]}\to \Sfrak_{\ZZ/n}$ can be understood as the orbifold fundamental group of $\M_{\ZZ/n}$ (which is usually denoted $\Mcal_{0,n}$).
\smallskip
 
Here is an alternative description of $\pi_{[n]}$ that is also relevant for his paper.
The Coxeter graph  of $\hat A_{n-1}$ is an $n$-gon (which we also denote by $\hat A_{n-1}$) and 
hence its automorphism group $\aut(\hat A_{n-1})$ is a dihedral group of order $2n$ (here we think of 
$\mu_n$ as the index 2 subgroup of $\aut(\hat A_{n-1})$ of its rotations). It is generated by the involution 
$\iota_0$ which fixes $\sigma_0$ and the involution 
$\iota_{0,1}$ which exchanges $\sigma_0$ and $\sigma_1$. The remaining generators 
$\sigma_1, \dots, \sigma_{n-1}$ resp.\ $\sigma_2, \dots, \sigma_{n-1}$ define a subdiagram of type 
$A_{n-1}$ resp.\ $A_{n-2}$ and generate Artin groups of that type (braid groups with  $n$ resp.\ $n-1$ strands). They therefore have Garside elements $\Delta_0$ resp.\  $\Delta_{0,1}$. We have $\Delta_0=\sigma_n\cdots\sigma_1\Delta_{0,1}$. These  Garside elements have the property that
left multiplication  with  a  generator $\sigma_i$ of this Artin subgroup produces the same  Garside element right multiplied with  $\sigma_{\iota_0(i)}$  resp.\  $\sigma_{\iota_{0,1}(i)}$. 
The imposed relations are equivalent to demanding that this remains true for the missing generators $\sigma_0$ resp.\  $\sigma_0, \sigma_1$ and that $\Delta_0^2\equiv\Delta^2_{0,1}\equiv 1$.
We state this as a proposition.

\begin{proposition}\label{prop:braidpres}
The group $\pi_{[n]}$ is also obtained as the quotient of $\ar(\hat A_{n-1})$ for which all the automorphisms induced by the automorphisms of  the Coxeter diagram become inner: the defining relations make  the involution which fixes $\sigma_0$ be defined by conjugation with the Garside element $\Delta_0$ of the $A_{n-1}$-subgraph  spanned by $\sigma_1, \dots, \sigma_{n-1}$ and the involution which exchanges 
$\sigma_0$ and $\sigma_1$ by conjugation with the Garside element $\Delta_{0,1}$ of the $A_{n-2}$-subgraph  spanned by $\sigma_2, \dots, \sigma_{n-1}$ and we also demand that  $\Delta_0^2\equiv\Delta_{0,1}^2\equiv 1$. \hfill $\square$
\end{proposition}

\begin{remark}[On the real structure and the base point]\label{rem:real base point}
The base point used for these presentations was 
represented by $\mu_n$.
As a subset of $\CC$ this is real, but its individual points are clearly not and so when used as a base point of $\M_n$ (via $i\in \ZZ/n\mapsto \zeta_n^i$), this might give the impression that it is not a real base point. We want however a real structure on  $\M_n$ for which this is the case.  

For this, we must assume that $C$ has a split real structure. This makes it $\RR$-isomorphic with $\PP^1$ so that $C(\RR)\cong \PP^1(\RR)$. This gives $\Cs_n(C)$ and $\Ss_n(C)$ a real structure and hence also their $\aut(C)$-orbit spaces $\M_n$ and $\M_{[n]}$.
Recall that we proceeded
to  choose a pair of points on $C$ and  then identified their complement with the multiplicative group $\CC^\times$. It was almost tacitly understood that these two points would be real and that the real structure on $\CC^\times$  is the usual one. However, the above desire prompts us make a different choice, namely to assume that  the pair of points we remove is a complex conjugate pair.  This makes the complement isomorphic with a  nonsplit copy of $\GG_m$, that is,  a copy of $\CC^\times$ on which complex conjugation  acts as $z\mapsto \bar z^{-1}$. Then $C(\RR)$ is represented by the unit circle of $\CC^\times$  and therefore to contains  $\mu_n$. 
The real injections $\{1,\dots , n\}\hookrightarrow C$ are those that take their value in this circle. These injections  
have several connected components, but if we divide out by $\aut(C)$, then each of these connected components becomes contractible. In particular, the connected component which contains $i\in\ZZ/n\mapsto \zeta_n^i$ is contractible. So  that subset may serve as our base point for $\M_n$ and $\M_{[n]}$ instead of $\mu_n$ and gives of course the same presentation. 
\end{remark}

\begin{proposition}%\label{prop:}
The kernel of the natural  homomorphism $\pi_{[n]}\to \Sfrak_{\ZZ/n}$  is  normally generated by $\sigma^2_0$. 
\end{proposition}
\begin{proof}
It suffices to show  that imposing the relation $\sigma_0^2\equiv 1$ on $\pi_{[n]}$ produces a quotient of  the permutation group $\Sfrak_{\ZZ/n}\cong \Sfrak_n$. 

Since $\mu_n$ permutes the generators $\sigma_i$ transitively, the relation $\sigma_0^2\equiv 1$ implies 
$\sigma_i^2\equiv 1$  for all $i$.  By Remark \ref{rem:presentation1}, $\pi_{[n]}$ is in fact a quotient of the subgroup 
generated by $\sigma_1,\dots, \sigma_{n-1}$. Hence imposing these  relations on $\pi_{[n]}$ must give us a quotient of 
the Coxeter group generated  $s_1, \dots, s_{n-1}$, i.e., a quotient of $\Sfrak_{\ZZ/n}$.
\end{proof}

\begin{remark}%\label{rem:}
We can more generally take any  integer $d>1$ and impose the relation $\sigma_i^d\equiv 1$ for all (or equivalently, some) $d$. The associated quotient $\pi_{[n]}\to \pi_{[n]}(d)$ has   sometimes 
a geometric interpretation. For example, if $d$ divides $n$, then any $n$-element subset  
$D\subset C$ defines a $\mu_d$-cover   $C_D(d)\to C$ which has order $n/d$ ramification over $D$.
This cover is unique up to a covering transformation. We think of the  embedding 
$\mu_d\subset \CC^\times$ as a character $\chi$. Then the corresponding eigenspace 
$H_1(C_D; \CC)^\chi$ is defined up to multiplication with $\mu_d$ and hence the projective space  
$\PP(H_1(C_D; \CC)^\chi)$ is intrinsic to $D$. This defines a 
flat projective bundle over $\M_{[n]}$. The monodromy of $\sigma_i$ is  acts as complex reflection of order $d$ and this representation factors through $\pi_{[n]}(d)$. For some of the Deligne-Mostow examples this is known to produce a faithful representation of $\pi_{[n]}(d)$. 
\end{remark}

\subsection{A line bundle and a central extension}
What  follows can be done for any integer $n\ge 3$, but the parity of $n$ makes for small  case distinctions. Since our interest here is the case when $n$ is even, we make this from now on a \emph{blanket assumption: 
$n=2m$ with $m$ an integer $\ge 2$}. 

Any  $D\in \sym^n(C)$ determines a line in  $\CC[V]_n$ (its `line of equations'), which we shall denote by 
$\ell_D$. This defines a line bundle over $\sym^n(C)$ with $\SL(V)$ action (when $n$ is even, this group  acts via $\PGL(V)=\aut(C)$). By passing to $\SL(V)$-orbit spaces, we get orbifold  line bundles 
\[
\begin{tikzcd}
\ell_{[n]}\arrow[r, hook]\arrow[d] &\ell^\st_{[n]}\arrow[d]\\
\M_{[n]}\arrow[r, hook]  &\M^\st_{[n]}
\end{tikzcd}
\]
%We denote by $\ell^\st_P\to \M^\st_P$ the pull-back of $\ell^\st_{[n]}$. 
It  is only defined up to the $\GG_m$-action: the reason is that $C$ determines $V$ only up to scalar multiplication. 
This ambiguity can be eliminated  by replacing $\ell_D$ by tensoring up with a one-dimensional vector space:
\[
\Ls_D:= \ell_D\otimes (\wedge^2V)^{\otimes m}, 
\]
for  the  $\GL(V)$-action now factors through 
 $\PGL(V)\cong \aut(C)$ and hence this line is canonically associated with $D$. We then let $\Ls^\st_{[n]}$ and $\Ls^\st_P$ and 
their restriction $\Ls_{[n]}$ and $\Ls_P$  have the obvious meaning. Here is a cohomological interpretation of $\Ls_D$.

\begin{lemma}\label{lemma:cohomological interpretation}
We have a natural isomorphism 
$\Lcal_D\cong  H^1(C, \theta^{m-1}_C(-D))$. 
\end{lemma}
\begin{proof}
Choose a generator $\mu\in \wedge^2V^*$ and  regard it as a translation invariant $2$-form on $V$. So if $G\in \CC[V]_2$ is homogeneous of degree $2$, then $\mu/G$ is rational $2$-form on $V$ which is homogeneous of degree zero.
By taking its residue at infinity we obtain a rational section of $\omega_C$ (where  $\omega_C$ stands for the sheaf of differentials on $C$) whose divisor is minus the degree 2 divisor defined by $G$. 
Hence,  if $F\in  \CC[V]_{n}$ defines the divisor $D$, then $\mu^{\otimes m}/F$ defines a rational section of $\omega_C^{\otimes m}$ whose divisor is $-D$, that is, a generator of 
$H^0(C,\omega_C^{\otimes m}(D))$. The (Serre) dual of the latter is $H^1(C, \theta^{\otimes (m-1)}_C(-D))$ and so 
a generator of the latter is represented by $F/\mu^{\otimes m}$. 
But clearly, $F/\mu^{\otimes m}$ is also a generator of $\ell_D\otimes (\wedge^2V)^{\otimes m}=\Ls_D$.
\end{proof}

Recall that $\CC[V]_n^\red\subset \CC[V]_n$ stands for  square free polynomials. We shall write  $\Es_{[n]}$  for its $\SL(V)$-orbit space. We regard  $\Es_{[n]}$ as an orbifold  $\GG_m$-bundle over $\M_{[n]}$. Its orbifold  fundamental group will be a central extension of the orbifold  fundamental group of  $\M_{[n]}$. Our aim is  to find  a presentation of this group in these terms and  
for this, we  emulate the procedure of the previous section. 

Let  $z_o\in V^*$ resp.\  $z_\infty\in V^*$ define $o\in C$ resp.\  $\infty\in C$. So $(z_o, z_\infty)$ is a basis of $V^*$ and 
$\CC[V]_n$  is the space of homogeneous polynomials in $z_o$ and $z_\infty$ of degree $n$.  The polynomials with the property 
that the coefficients of $z_o^n$  and $z_\infty^n$ are both equal to $1$  make up an affine subspace  
$\AA_{n-1}\subset \CC[V]_n$ of dimension $n-1$ of which every element can be written as a product
$(t_0z_o-z_\infty)\cdots (t_{n-1}z_o-z_\infty)$, with $t_0\dots t_{n-1}=1$. This  identifies $\AA_{n-1}$
with $\Sfrak_n\bs T_{n-1}$ and hence $\AA_{n-1}^\red:=\AA_{n-1}\cap \CC[V]^\red_n$ with  $\Sfrak_n\bs T^\reg_{n-1}$ (for this means that  $t_0, \dots, t_{n-1}$ are pairwise distinct). The base point $[\mu_n]$ of  $\Sfrak_n\bs T^\reg_{n-1}$ then corresponds with 
$z_o^n-z_\infty^n\in  \CC[V]_n^\red$. We denote the  image of $z_o^n-z_\infty^n$ in various subquotients of $\CC[V]_n^\red$ by $[z_o^n-z_\infty^n]$.
\smallskip

The action of $\PGL(V)\cong \aut(C)$ on $\CC[V]_n$ evidently commutes with the $\GG_m$-action  defined by scalar multiplication. 
Recall that we identified $T$ with $\GG_m$ via its action on $T_oC$. Via the identification,   $(\lambda,t)\in \GG_m\times T$  takes  the monomial  $z_0^{n-k}z_\infty^k$ to $\lambda t^{k-m}z_o^{n-k}z_\infty^k$ and so this  action on $\CC[V]_n$ is faithful. 

Let  $\CC[V]^\circ_n\subset \CC[V]_n$ be the open subset where  the coefficients of both $z_o^n$  and $z_\infty^n$ are  nonzero.
We denote the $T$-orbit space of $ \CC[V]_n^\red $ resp.\  $\CC[V]^\circ_n\cap \CC[V]_n^\red$ by
 $\Es_n(C)$ resp. $\Es_n(C^\circ)$.
So these orbit spaces inherit a faithful residual action by $\GG_m$. The $\GG_m$-quotient of $\Es_n(C^\circ)\subset \Es_n(C)$ is  $\M_n(C^\circ)\subset\M_n(C)$, so that we have  a commutative diagram of $\GG_m$-bundles:
\begin{equation}\label{eqn:Gm diagram}
\begin{tikzcd}
\Es_n(C^\circ) \arrow[r, hook]\arrow[d] &\Es_n(C)\arrow[d]\arrow[r]  &\Es_{[n]}\arrow[d]\\
\M_n(C^\circ)\arrow[r, hook] &\M_n(C) \arrow[r]  &\M_{[n]}
\end{tikzcd}
\end{equation}
We use this to determine the orbifold fundamental  group of $\Es_{[n]}$.

\begin{lemma}\label{lemma:section} 
Every $\GG_m\times T$-orbit in $\CC[V]^\circ_n\cap \CC[V]_n^\red$ meets $\AA_{n-1}^\red$ in a $\mu_{n}$-orbit, where $\mu_n$ acts via the isomorphism $\Sfrak_n\bs T_{n-1}^\reg\cong \AA_{n-1}^\red$. 
This defines a section of $\Es_n(C^\circ)\to  \M_n(C^\circ)$  which  trivializes that  bundle.
\end{lemma}
\begin{proof}
Let $u=u_0z_o^n+\cdots +u_n z_\infty^n\in \CC[V]^\circ_n$ , so with $(u_0,u_n)\not=(0,0)$.
Note that $(\lambda,t)\in  \GG_m\times T$ acts on $z_o^n$ resp.\ $z_\infty^n$ as multiplication with $\lambda t^{-m}$ resp.\  $\lambda t^{m}$. It is clear that we can find $(\lambda,t)$ such that 
$(\lambda t^{-m}u_0, \lambda t^{m}u_1)=(1,1)$. This proves that the $\GG_m\times T$-orbit of $u$ meets $\AA_{n-1}$.
It is  clear that if  $u\in \CC[V]_n^\red$, then this intersection lies in $\AA_{n-1}^\red$. 

To see that this intersection  is in fact a $\mu_n$-orbit, note 
that $(\lambda t^{-m}, \lambda t^{m})=(1,1)$ if and only if $(\lambda,t)=((-1)^k, \zeta_n^k)$ for some $k\in\ZZ$ 
and that  $((-1)^k, \zeta_n^k)$  takes the the monomial  $z_0^{n-k}z_\infty^k$ to $\zeta_n^kz_o^{n-k}z_\infty^k$. So  
this is an orbit of a $\mu_n$-action on $\AA_{n-1}$ which covers the $\mu_n$-action on $\Ss_n(C^\circ)$. This implies that the  image of $\AA_{n-1}^\red$  in $\Es_n(C^\circ) $ defines a section of $\Es_n(C^\circ)\to  \M_n(C^\circ)$. 
\end{proof}

The diagram  \eqref{eqn:Gm diagram} gives us a  corresponding diagram of central extensions of the orbifold fundamental groups of the bottom row by the  infinite cyclic group $\pi_1(\GG_m,1)$. A generator $\sigma$ of $\pi_1(\GG_m,1)$  is  represented by the loop $\theta\in [0, 2\pi]\mapsto e^{\sqrt{-1}\theta}\in \GG_m$.  
%We write this group multiplicatively, and so it is equal to $\sigma^\ZZ$. 

By Lemma \ref{lemma:pi1diagram}, the right bottom map of diagram \eqref{eqn:Gm diagram} induces an isomorphism on orbifold fundamental groups. Hence the same is true for the right top map. This brings us in a situation similar to that of the previous section. Lemma \ref{lemma:section} and the isomorphism \eqref{eqn:artinextended} identify the orbifold fundamental group of $\Es_n(C^\circ)$ with $\sigma^\ZZ\times (\ar(\hat A_{n-1})\rtimes \mu_n)$. Note that
$\Es_n(C)\ssm \Es_n(C^\circ)$ is a divisor with two irreducible components, consisting of the $T$-orbit space of the polynomials for which the coefficient of $z_o^n$  resp.\ $z_\infty^n$ is zero (i.e,  divisible by $z_\infty$ resp.\ $z_o$). So  a simple loop $\hat\delta_\infty$ resp. $\hat\delta_0$ around  them, when expressed as a word in  $\sigma^\ZZ\times (\ar(\hat A_{n-1})\rtimes \mu_n)$  produces the relations that define the orbifold fundamental group of $\Es_n(C)$ and hence of $\Es_{[n]}$.

\begin{theorem}\label{thm:central ext}
The orbifold fundamental group $\pi_1^\orb( \Es_{[n]}, [z_o^n-z_\infty^n])$ is the quotient of 
$\ar(A_1\times \hat A_{n-1})\rtimes \mu_n$ by the relations 
$\sigma\equiv \zeta_n^{-1}\sigma_1\sigma_2\cdots \sigma_{n-1}$
and $\sigma\equiv \zeta_n\sigma_{n-1}\sigma_{n-2}\cdots \sigma_1$.
\end{theorem}
\begin{proof}
We can take for $\hat\delta_\infty$ the loop 
\[
t\in [0,1]\mapsto (z_o -\alpha_R(t)z_\infty)(z_o-\zeta_nz_\infty) \dots (z_o-\zeta_n^{n-1}z_\infty),
\]
where $\alpha_R (t)$ is a  loop in $\CC\ssm\{\zeta_n, \dots, \zeta_n^{n-1}\}$ based at $1$ as in the proof of Corollary \ref{cor:presentation1}:  it first traverses the interval $[1,R]$ (with $R\gg 1$), then the full circe $|z|=R$ in the counterclockwise direction and then returns along $[R,1]$. It is straightforward to check that this loop represents $\sigma^k\zeta_n\sigma_{n-1}\sigma_{n-2}\cdots \sigma_1$.
Likewise, we can take for $\hat\delta_0$ the loop 
\[
t\in [0,1]\mapsto  (\alpha_R(t)z_o -z_\infty)(z_o-\zeta_nz_\infty) \dots (z_o-\zeta_n^{n-1}z_\infty),
\]
This represents $\sigma^k\zeta_n^{-1}\sigma_1\sigma_2\cdots \sigma_{n-1}$.
\end{proof}

\begin{proposition}\label{prop:alternative presentation}
The group $\pi_1^\orb( \Es_{[n]}, [z_o^n-z_\infty^n])$ is also a quotient of  the Artin group 
$\ar(A_1\times \hat A_{n-1})$ for which all the Coxeter diagram automorphisms  become inner:  the rotation which take $\sigma_i$ to $\sigma_{i+1}$ resp.\ $\sigma_{i-1}$  is induced by conjugation with  $\sigma _1\cdots \sigma_{n-1}$ resp.\
$\sigma_{n-1}\cdots \sigma_1$ and the involution which exchanges $\sigma_i$ and $\sigma_{-i}$ by is induced by conjugation with the Garside element of  the $A_{n-1}$-subgraph  spanned by $\sigma_1, \dots, \sigma_{n-1}$. Furthermore, $(\sigma_1\cdots \sigma_{n-1})^n\equiv(\sigma_{n-1}\cdots \sigma_1)^n\equiv\sigma^n$. 
\end{proposition}

\end{document}